\documentclass[reqno, 11pt, a4paper]{amsart}

\usepackage{latexsym,ifthen,xspace}
\usepackage{amsmath,amssymb,amsthm}
\usepackage{enumerate, calc}
\usepackage[colorlinks=true, pdfstartview=FitV, linkcolor=blue,
  citecolor=blue, urlcolor=blue, pagebackref=false]{hyperref}
\usepackage[text={33pc,605pt},centering]{geometry}    
\usepackage{graphicx}

\newtheorem{theorem}{Theorem}[section]
\newtheorem{lemma}[theorem]{Lemma}
\newtheorem{prop}[theorem]{Proposition}
\newtheorem{assumption}[theorem]{Assumption}
\newtheorem{corro}[theorem]{Corollary}

\theoremstyle{definition}
\newtheorem{definition}[theorem]{Definition}

\theoremstyle{remark}
\newtheorem{remark}[theorem]{Remark}

\numberwithin{equation}{section}

\DeclareMathAlphabet{\mathsl}{OT1}{cmss}{m}{sl}
\SetMathAlphabet{\mathsl}{bold}{OT1}{cmss}{bx}{sl}

\newcommand{\al}{\ensuremath{\alpha}}
\newcommand{\be}{\ensuremath{\beta}}
\newcommand{\ga}{\ensuremath{\gamma}}
\newcommand{\de}{\ensuremath{\delta}}

\renewcommand{\th}{\ensuremath{\theta}}

\newcommand{\ka}{\ensuremath{\kappa}}

\newcommand{\si}{\ensuremath{\sigma}}

\newcommand{\om}{\ensuremath{\omega}}

\newcommand{\ve}{\ensuremath{\varepsilon}}

\newcommand{\Si}{\ensuremath{\Sigma}}

\newcommand{\Om}{\ensuremath{\Omega}}
\newcommand{\cB}{\ensuremath{\mathcal B}}

\newcommand{\cE}{\ensuremath{\mathcal E}}
\newcommand{\cF}{\ensuremath{\mathcal F}}

\newcommand{\cL}{\ensuremath{\mathcal L}}

\newcommand{\cP}{\ensuremath{\mathcal P}}

\newcommand{\bbN}{\ensuremath{\mathbb N}} 
 
\newcommand{\bbP}{\ensuremath{\mathbb P}} 
 
\newcommand{\bbR}{\ensuremath{\mathbb R}}

\newcommand{\bbZ}{\ensuremath{\mathbb Z}} 
\newcommand{\md}{\ensuremath{\mathrm{d}}}
\newcommand{\mD}{\ensuremath{\mathrm{D}}}

\newcommand{\scpr}[3]{%
  \ensuremath{%
    \left\langle
      #1, #2
    \right\rangle_{\raisebox{-0ex}{$\scriptstyle \ell^{\raisebox{.1ex}{$\scriptscriptstyle 2$}} (#3)$}}
  }
}
\newcommand{\norm}[3]{%
   \ensuremath{%
     \mathchoice{\big\lVert #1 \big\rVert}
     {\lVert #1 \rVert}
     {\lVert #1 \rVert}
     {\lVert #1 \rVert}_{\raisebox{-.0ex}{$\scriptstyle \ell^{\raisebox{.2ex}{$\scriptscriptstyle #2$}} (#3)$}}
   }
}

\newcommand{\Norm}[2]{%
  \ensuremath{%
    \mathchoice{\big\lVert #1 \big\rVert}
     {\lVert #1 \rVert}
     {\lVert #1 \rVert}
     {\lVert #1 \rVert}_{\raisebox{-.0ex}{$\scriptstyle #2$}}
  }
}

\DeclareMathOperator{\mean}{\mathbb{E}}
\DeclareMathOperator{\Mean}{\mathrm{E}}
\DeclareMathOperator{\prob}{\mathbb{P}}
\DeclareMathOperator{\Prob}{\mathrm{P}}

\DeclareMathOperator{\supp}{\mathrm{supp}}
\DeclareMathOperator{\sign}{\mathrm{sign}}

\newcommand{\ldef}{\ensuremath{\mathrel{\mathop:}=}}
\newcommand{\rdef}{\ensuremath{=\mathrel{\mathop:}}}

\def\indicator{{\mathchoice {1\mskip-4mu\mathrm l}%
{1\mskip-4mu\mathrm l}{1\mskip-4.5mu\mathrm l}%
{1\mskip-5mu\mathrm l}}}

\begin{document}

\title[Quenched CLT for RCM with ergodic conductances]{Invariance principle for the random conductance model in a degenerate ergodic environment}


\author{Sebastian Andres}
\address{Rheinische Friedrich-Wilhelms Universit\"at Bonn}
\curraddr{Endenicher Allee 60, 53115 Bonn}
\email{andres@iam.uni-bonn.de}
\thanks{}

\author{Jean-Dominique Deuschel}
\address{Technische Universit\"at Berlin}
\curraddr{Strasse des 17. Juni 136, 10623 Berlin}
\email{deuschel@math.tu-berlin.de}
\thanks{}

\author{Martin Slowik}
\address{Technische Universit\"at Berlin}
\curraddr{Strasse des 17. Juni 136, 10623 Berlin}
\email{slowik@math.tu-berlin.de}
\thanks{}



\date{\today}

\dedicatory{}

\begin{abstract}
  We study a continuous time random walk, $X$, on $\bbZ^d$ in an environment of random conductances taking values in $(0, \infty)$.  We assume that the law of the conductances is ergodic with respect to space shifts.  We prove a quenched invariance principle for $X$ under some moment conditions on the environment.  The key result on the sublinearity of the corrector is obtained by Moser's iteration scheme. 
\end{abstract}

\maketitle

\tableofcontents

\section{Introduction}

\subsection{The model}
Consider the $d$-dimensional Euclidean lattice, $(V_d, E_d)$, for $d \geq 2$.  The vertex set, $V_d$, of this graph equals $\bbZ^d$ and the edge set, $E_d$, is given by the set of all nonoriented nearest neighbor bonds, i.e.\ $E_d \ldef \{ \{x,y\}: x,y \in \bbZ^d, |x-y|=1\}$. 

Let $(\Om, \cF) = \big((0,\infty)^{E_d}, \cB((0,\infty))^{\otimes\, E_d}\big)$ be a measurable space.  Assume further that the graph $(V_d, E_d)$ is endowed with positive weights, that is, we consider a family $\om = \{\om(e) : e \in E_d\} \in \Om$.  We refer to $\om(e)$ as the \emph{conductance} on an edge, $e$.  We will henceforth denote by $\prob$ a probability measure on $(\Om, \cF)$, and we write $\mean$ to denote the expectation with respect to $\prob$.  To lighten notation, for any $x,y \in \bbZ^d$, we set
\begin{align*}
  \om_{xy} \;=\; \om_{yx} \;\ldef\; \om(\{x,y\}),
  \quad \forall\, \{x,y\} \in E_d,
  \mspace{28mu}
  \om_{xy} \;\ldef\; 0, \quad \forall\, \{x,y\} \not\in E_d.
\end{align*}
A \emph{space shift} by $z \in \bbZ^d$ is a map $\tau_z\!: \Om \to \Om$
\begin{align}\label{eq:def:space_shift}
  (\tau_z \om)_{xy} \;\ldef\; \om_{x+z,y+z},
  \qquad \forall\, \{x,y\} \in E_d.
\end{align}
The set $\big\{\tau_x : x \in \bbZ^d\big\}$ together with the operation $\tau_{x} \circ \tau_{y} \ldef \tau_{x+y}$ defines the \emph{group of space shifts}.

We will study the nearest neighbor \emph{random conductance model}.  For any fixed realization $\om$, it is a \emph{reversible} continuous time Markov chain, $X = (X_t\!: t \geq 0 )$, on $\bbZ^d$ with generator $\cL^{\om}$ acting on bounded functions $f\!: \bbZ^d \to \bbR$ as
\begin{align*}
  \big(\cL^{\om} f)(x)
  \;=\;
  \sum_{y \in \bbZ^d} \om_{xy}\, \big(f(y) - f(x)\big).
\end{align*}
We denote by $\Prob_x^\om$ the law of the process starting at the vertex $x \in \bbZ^d$.  The corresponding expectation will be denoted by $\Mean_x^\om$.  Setting $\mu^{\om}(x) \ldef \sum_{y \in \bbZ^d} \om_{xy}$ and $p^{\om}(x,y) \ldef \om_{xy} / \mu^{\om}(x)$, this random walk waits at $x$ an exponential time with mean $1/ \mu^{\om}(x)$ and chooses its next position $y$ with probability $p^{\om}(x,y)$.  Since the law of the waiting times depends on the location, $X$ is also called the \emph{variable speed random walk} (VSRW).

We denote by $p^{\om}(t,x,y) \ldef \Prob_{x}^{\om}[X_t = y]$ for $x, y \in \bbZ^d$ and $t \geq 0$ the transition densities with respect to the counting measure.  As a consequence of \eqref{eq:def:space_shift} we have
\begin{align*}
  p^{\tau_z \om}(t, x, y)
  \;=\;
  p^{\om}(t, x+z, y+z).
\end{align*}
\begin{assumption} \label{ass:environment}
  Assume that $\prob$ satisfies the following conditions:
  \begin{enumerate}[ (i) ]
    \item $\prob\!\big[ 0 < \om(e) < \infty\big] = 1$ and $\mean\!\big[\om(e)\big] < \infty\, $ for all $e \in E_d$.
    \item $\prob$ is ergodic with respect to translations of $\bbZ^d$, that is, $\prob \circ\, \tau_x^{-1} \!= \prob\,$ for all $x \in \bbZ^d$ and $\prob[A] \in \{0,1\}\,$ for any $A \in \cF$ such that $\tau_x(A) = A\,$ for all $x \in \bbZ^d$.
  \end{enumerate}
\end{assumption}

\subsection{Results}
We are interested in the $\prob$-almost sure or quenched long range behavior, in particular in obtaining a quenched functional central limit theorem for the process $X$ in the sense of the following definition.
\begin{definition} \label{def:QFCLT}
  Set $X_t^{(n)} \ldef \frac{1}{n} X_{n^2 t}$, $t \geq 0$. We say that the \emph{Quenched Functional CLT} (QFCLT) or \emph{quenched invariance principle} holds for $X$ if for $\prob$-a.e.\ $\om$ under $\Prob_{\!0}^\om$, $X^{(n)}$ converges in law to a Brownian motion on $\bbR^d$ with covariance matrix $\Si^2 = \Si \cdot \Si^T$.  That is, for every $T > 0$ and every bounded continuous function $F$ on the Skorohod space $D([0,T], \bbR^d)$, setting $\psi_n = \Mean_0^{\om}[F(X^{(n)})]$ and $\psi_\infty = \Mean_0^{\mathrm{BM}}[F(\Si \cdot W)]$ with $(W, \Prob_{\!0}^{\mathrm{BM}})$ being a Brownian motion started at $0$, we have that $\psi_n \rightarrow \psi_\infty$ $\prob$-a.s.
\end{definition}
As our main result we establish a QFCLT for $X$ under some additional moment conditions on the conductances.
\begin{theorem}\label{thm:main}
  Suppose that $d \geq 2$ and Assumption \ref{ass:environment} holds.  Let $p,q \in (1, \infty]$ be such that $1/p + 1/q < 2/d$ and assume that
  \begin{align}\label{eq:moment_condition}
    \mean\!\big[(\om(e))^p\big] < \infty
    \quad \text{and} \quad
    \mean\!\big[(1/\om(e))^q\big] < \infty
  \end{align}
  for any $e \in E_d$.  Then, the QFCLT holds for $X$  with a deterministic nondegenerate covariance matrix $\Si^2$.
\end{theorem}
\begin{remark}
If the law $\prob$ of the conductances is also invariant under symmetries of $\bbZ^d$, then the limiting covariance matrix $\Si^2$ must be invariant under symmetries as well, so $\Si^2$ is of the form $\Si^2 = \si^2 I$ for some $\sigma>0$. (Here  $I$ denotes the identity matrix.)
\end{remark}
\begin{remark} \label{rem:csrw}
  Given a speed measure $\pi^\om\!: \bbZ^d \to (0,\infty)$ satisfying $\pi^\om(x) = \pi^{\tau_x \om}(0)$ and $0 < \mean[\pi^{\om}(0)] < \infty$, one can also consider the process, $Y = (Y_t\!: t \geq 0)$ on $\bbZ^d$ that is defined by a time change of $X$, that is, $Y_t \ldef X_{a_t}$ for $t \geq 0$, where $a_t \ldef \inf\{s \geq 0 : A_s > t\}$ denotes the right continuous inverse of the functional
  \begin{align*}
    A_t \;=\; \int_0^t \pi^\om(X_s) \, \md s, \qquad t \geq 0.
  \end{align*}
  Its generator is given by
  \begin{align*}
    \cL_Y^{\om} f(x)
    \;=\;
    \sum_{y \in \bbZ^d} \frac{\om_{xy}}{\pi^{\om}(x)} \big(f(y)-f(x)\big).
  \end{align*}
  Suppose that $X$ satisfies an invariance principle, that is, the rescaled process converges to a Brownian motion on $\bbR^d$ with covariance matrix $\Si^2_X$. As it was shown in \cite[Section 6.2]{ABDH12}, the process $Y$ satisfies an invariance principle as well.  In this case, the covariance matrix of its limiting Brownian motion is given by $\Si_Y^2 = \mean[\pi^{\om}(0)]^{-1} \Si_X^2$.

  A natural choice for the speed measure would be $\pi^{\om} = \mu^{\om}$. In such a case, $Y$ is called \emph{constant speed random walk (CSRW)}.  In contrast to the VSRW $X$ whose waiting time at any site $x \in \bbZ^d$ depends on $x$, the process $Y$ waits at each site an exponential time with mean $1$.
\end{remark}
\begin{remark}
  Note that Assumption~\ref{ass:environment} (i) ensures the stochastic completeness of the process $X$, that is, it does not explode in finite time almost surely.
\end{remark}
Invariance principles for the random conductance model have been studied by a number of different authors under various restrictions on the law of the environment.  A weak FCLT, that is, where the convergence of $\psi_n$ to $\psi_\infty$ in Definition~\ref{def:QFCLT} only takes place in $\bbP$-probability, has been proved already in \cite{dMFGW89} (cf.\ also \cite{KV86}) for general ergodic environments under the first moment condition $\mean[\om(e)] < \infty$.  However, it took quite some time to extend this result to $\prob$-almost sure convergence in the special case of a uniformly elliptic environment with  conductances which are bounded both from above and below, that is, $\prob[1/c \leq \om(e) < c ] = 1$ for some $c > 0$; see \cite{SS04}.

In the very special case of i.i.d.\ conductances, that is when $\prob$ is a product measure, it turns out that no moment conditions are required for the QFCLT, cf. \cite{ABDH12}, see also \cite{BB07,MP07} for the corresponding supercritical percolation model and \cite{BP07,Ma08,BD10} for similar results.  In the setting of balanced random walks in random environment, a similar type of result holds, namely the QFCLT is true for general ergodic environments under some moment conditions \cite{GZ12}, whereas in the i.i.d.\ case no ellipticity is needed \cite{BD13}. Finally, for a recent result on random walks under random conductances on domains with boundary, we refer to \cite{CCK13}.

In the case of a general ergodic environment, it is clear that some moment conditions are needed, in particular, in \cite{BBT13, BBT13a} Barlow, Burdzy and Timar give an example on $\bbZ^2$, for which the weak FCLT holds but the QFCLT fails.  In their model, \eqref{eq:moment_condition} is assumed for $p, q \in (0,1)$ but $\mean[\om(e)] = \infty$ and $\mean[1/\om(e)] = \infty$.  Recently, in \cite{Bi11} Biskup proved the  QFCLT for the special  case $d = 2$ under the moment condition \eqref{eq:moment_condition} with $p = q = 1$.  We believe that this is a natural optimal condition for general ergodic environments. Nevertheless, the proof relies on arguments -- inspired by \cite{BB07} for the percolation case -- which only work in the very special two-dimensional case (see below) and do not seem to be extendable to higher dimensions.  Recently, in \cite{PRS13}, a QFCLT has been proven for simple random walks on percolation models on $\bbZ^d$ with long-range correlations such a!
 s random interlacements or the level set of the Gaussian free field.

Finally, let us remark that under the moment conditions \eqref{eq:moment_condition} the result in \cite{BS12arx} guarantees that the random walk $X$ has vanishing speed.

\subsection{The Method}
The main ingredient to prove a QFCLT is to introduce harmonic coordinates, that is one constructs a \emph{corrector} $\chi\!:\Om \times \bbZ^d \to \bbR^d$ such that
\begin{align*}
  \Phi(\om, x) \;=\; x - \chi(\om, x)
\end{align*}
is an $\cL^\om$-harmonic function, that is, for $\prob$-a.e.\ $\om$ and every $x \in \bbZ^d$,
\begin{align*}
  \cL^\om \Phi(\om, x)
  \;=\;
  \sum_{y} \om_{xy} \big( \Phi(\om, y) - \Phi(\om, x) \big)
  \;=\;
  0.
\end{align*}
This can be rephrased by saying that $\chi$ is a solution of the Poisson equation
\begin{align} \label{eq:pois_chi}
  \cL^{\om} u \;=\; \nabla^* V^{\om},
\end{align}
where $V^{\om}\!:E_d \to \bbR^d$ is the local drift given by $V^{\om}(x,y) \ldef \om_{xy} \, (y-x)$ and $\nabla^*$ denotes the divergence operator associated with the discrete gradient.  

Moreover, the corrector $\chi$ needs to be shift invariant in the sense that it satisfies $\bbP$-a.s.\ the following \emph{cocycle property}:
\begin{align*}
  \chi(\om, x + y) - \chi(\om, x)
  \;=\;
  \chi(\tau_x\om, y),
  \qquad x,y \in \bbZ^d.
\end{align*}
The construction of the corrector follows from a simple projection argument of the trivial cocycle $\Pi(\om, x) = x$ under the first moment condition $\mean[\om(e)]<\infty$.  The $\cL^\om$-harmonicity of $\Phi$ implies that
\begin{align*}
  M_t \;=\; X_t - \chi(\om, X_t)
\end{align*}
is a martingale under $\Prob_{\!0}^{\om}$ for $\prob$-a.e.\ $\om$, and a QFCLT for the martingale part $M$ can be easily shown by standard arguments.  We thus get a QFCLT for $X$ once we verify that $\prob$-almost surely the corrector is sublinear:
\begin{align} \label{eq:sublin_intro}
  \lim_{n \to \infty}  \max_{|x| \leq n} \frac{ \left| \chi(\om,x) \right|}{n}
  \;=\;
  0.
\end{align}
This control on the corrector implies that for any $T>0$ and $\prob$-a.e\ $\om$
\begin{align*}
  \sup_{0 \,\leq\, t \,\leq\, T}\,
    \frac{1}{n}\, \Big| \chi\big(\om, n\, X_{t}^{(n)}\big) \Big|
    \;\underset{n \to \infty}{\longrightarrow}\;
    0 
    \quad \text{ in $\Prob_{\!0}^\om$-probability}
\end{align*}
(see Proposition~\ref{prop:contr_corr} below).  Combined with the QFCLT for the martingale part this gives Theorem~\ref{thm:main}.

The main challenge in the proof of the QFCLT is to prove \eqref{eq:sublin_intro}.  Using the cocycle property  and ergodicity of the environment, it is easy to verify that the corrector is sublinear along each line parallel to the axis:
\begin{align*}
  \lim_{n\to\infty} \frac{1}{n} \chi(\om, n\, e_i)
  \;=\;
  0,
  \qquad i = 1, \ldots, d
  \qquad \prob\text{-a.s.}
\end{align*}
In dimension $d = 2$, this and the fact that $\chi$ solves the Poisson equation \eqref{eq:pois_chi} suffices for \eqref{eq:sublin_intro} to hold, cf. \cite{BB07} and \cite{Bi11}.  In higher dimensions $d \geq 3$, heat kernel estimates for the transition density, cf.\ \cite{De99, Ba04, MR04}, have been used so far in the proofs of the quenched functional CLT, cf.\ \cite{ABDH12, BD10, BB07, BP07, Ma08, MP07, SS04}.  This method is very performing, provided one has a good control on the geometry of the "bad" configurations which are the connected components of very low or very high conductances.  This is the case in the i.i.d.\ setting and would probably work under good mixing properties, but this is not likely to be  the case in a general ergodic environment.

An alternative proof would be to get $L^p(\prob)$ estimates of the corrector for $p > d$  and use ergodic theory for cocycle established in \cite{BD91}.  However, so far these estimates have only been derived for nicely mixing elliptic environments \cite{GO11}.

Motivated by the method of \cite{FK97, FK99}, where diffusions in divergence form in a random environment are considered, we present in this paper a control of the corrector using the Moser iteration.  Moser's iteration is based on two main ideas: the Sobolev inequality (cf.\ Proposition~\ref{prop:sob} below) which allows to control the $\ell^r$-norm with $r = r(d) = d/(d-2) > 1$ in terms of the Dirichlet form, and a control of the Dirichlet form of the solution of the Poisson equation \eqref{eq:pois_chi} (see Lemma~\ref{lem:moser:DF} below).  In the uniformly elliptic case, this is rather standard.  In our case where the conductances are  unbounded from above and below, we need to work with a dimension dependent weighted Sobolev inequality, which we obtain from H\"older's inequality.  That is, we replace the coefficient $r(d)$ by
\begin{align*}
  r(d,p,q) \;=\; \frac{d - d/p}{(d-2) + d/q}
\end{align*}
(cf.\ Remark \ref{rem:sob} below).  For the Moser iteration, we need $r(d,p,q) > 1$, of course, which is equivalent to $1/p + 1/q < 2/d$ appearing in \eqref{eq:moment_condition}.

Although we do not quite recover Biskup's optimal result in $d = 2$, we believe that our method is very efficient for the following reasons:  First, we present a proof in higher dimensions which does not rely on heat kernel estimates. Second, our method is very robust and can be extended to both the random graph setting (cf.\ \cite{PRS13}), provided some  a priori isoperimetric inequality, and also for the time-dynamic conductance models, cf.\ \cite{An12}.

Recently, in \cite{BM13}, Ba and Mathieu have established a QFCLT for diffusions in $\bbR^d$ with a locally integrable periodic potential.  Their approach is also based on a Sobolev-type inequality, where the sublinearity of the corrector is only obtained along the path of the process. 

Let us remark that our result applies to a random conductance model given by
\begin{align*}
  \om_{xy} = \exp(\phi(x) + \phi(y)),
  \qquad \{x,y\} \in E_d,
\end{align*}
where  for $d \geq 3$, $\{\phi(x) : x \in \bbZ^d\}$ is the discrete massless Gaussian free field, cf.\ \cite{BS12}.  In this case, the moment condition \eqref{eq:moment_condition} holds for any $p, q \in (0, \infty)$, of course.

Finally, in \cite{ADS13a}, we will apply  the adaptations of the Moser iteration technique in this work to derive both elliptic and parabolic Harnack inequalities under the assumptions of Theorem~\ref{thm:main}, which can be used to derive a quenched local limit theorem.

The paper is organized as follows:  In Section~\ref{sec:QFCLT}, we prove our main result, where we first recall the construction of the corrector.  Then we prove the sublinearity of the corrector \eqref{eq:sublin_intro} and complete the proof of the QFCLT.  In order to obtain \eqref{eq:sublin_intro} we first show that in a space-averaged $\ell^1$-norm the rescaled corrector vanishes in the limit.  We then need the Moser iteration technique, which provides us with a control on the maximum norm in terms of any  averaged $\ell^p$-norm.  This estimate is proven in a more general context in Section~\ref{sec:mos_it}.  Finally, the Appendix contains a collection of some elementary estimates needed in the proofs.

Throughout the paper, we write $c$ to denote a positive constant which may change on each appearance.  Constants denoted $C_i$ will be the same through each argument.

\section{Quenched invariance principle}
\label{sec:QFCLT}
\subsection{Harmonic embedding and the corrector}
\label{subsec:corr}
In this subsection, we first prove the existence of a corrector to the process $X$ such that $M_t = X_t - \chi(\om, X_t)$ is a martingale under $\Prob_{\!0}^{\om}$ for $\prob$ a.e.\ $\om$.  In a second step, we show an invariance principle for the martingale part.

\begin{definition}
  A measurable function, also called random field, $\Psi\!: \Omega \times \bbZ^d \to \bbR$ satisfies the \emph{cocycle property} if for $\prob$-a.e.\ $\om$, it holds that
  \begin{align*} 
    \Psi(\tau_x \om,y-x)
    \;=\;
    \Psi(\om,y) \,-\, \Psi(\om,x),
    \qquad \text{for all } x, y \in \bbZ^d.
  \end{align*}
  We denote by $L^2_{\mathrm{cov}} $ the set of functions $\Psi\!: \Om \times \bbZ^d \to \bbR$ satisfying the cocycle property such that
  \begin{align*}
    \Norm{\Psi}{L^2_{\mathrm{cov}}}^2
    \;\ldef\;
    \mean\!\Big[{\textstyle \sum_{x \in \bbZ^d}}\; \om_{0x}\, \Psi^2(\om, x) \Big] 
    \;<\; \infty.
  \end{align*}
\end{definition}
It can easily be checked that $L_{\mathrm{cov}}^2$ is a Hilbert space. 
\begin{lemma} \label{basicG}
  Consider a $\Psi \in L_{\mathrm{cov}}^2$. Then:
  \begin{enumerate}[(i)]
    \item for $\prob$-a.e.\ $\om$, $\Psi(\om,0) = 0$ and $\Psi(\tau_x \om,-x) = -\Psi(\om,x)$ for all $x \in \bbZ^d$.
    \item If  $x_0, x_1, \dots, x_n \in \bbZ^d$ then
      \begin{align} \label{eq:chain}
        \sum_{i=1}^n \Psi( \tau_{x_{i-1}} \om, x_i-x_{i-1})
        \;=\;
        \Psi(\om,x_n) - \Psi(\om,x_0). 
      \end{align}
  \end{enumerate}
\end{lemma}

\begin{proof}
  (i) follows immediately from the definition.  (ii) Since $\Psi$ satisfies the cocycle property, $\Psi(\tau_{x_{i-1}} \om, x_i-x_{i-1}) = \Psi(\om, x_i) - \Psi(\om,x_{i-1})$ and \eqref{eq:chain}. 
\end{proof}
Recall that $\Om = (0, \infty)^{E_d}$.  We say a function $\phi\!: \Om \to \bbR$ is \emph{local} if it only depends on the value of $\om$ at a finite number of edges.  We associate to $\phi$ a (horizontal) gradient $\mD \phi\!: \Om \times \bbZ^d \to \bbR$ defined by
\begin{align*}
  \mD \phi (\om,x)
  \;=\;
  \phi(\tau_x \om) - \phi(\om),
  \qquad x \in \bbZ^d.
\end{align*}
Obviously, if the function $\phi$ is bounded, $\mD \phi$ is an element of $L_{\mathrm{cov}}^2$. Following \cite{MP07}, we introduce an orthogonal decomposition of the space $L_{\mathrm{cov}}^2$.  Set
\begin{align*}
  L_{\mathrm{pot}}^2
  \;=\;
  \mathop{\mathrm{cl}}
  \big\{ \mD \phi \mid \phi\!: \Om \to \bbR\; \text{ local} \big\}
  \;\text{ in }\;  L_{\mathrm{cov}}^2,
\end{align*}
being the subspace of ''potential'' random fields and let $L_{\mathrm{sol}}^2$ be the orthogonal complement of $L_{\mathrm{pot}}^2$ in $L_{\mathrm{cov}}^2$ called ''solenoidal'' random fields.

In order to define the corrector, we introduce the \emph{position field} $\Pi\!: \Om \times \bbZ^d \to \bbR^d$ with $\Pi(\om,x) = x$.  We  write $\Pi_j$ for the $j$th coordinate of $\Pi$.  Since $ \Pi_j(\om, y-x) = \Pi_j(\om, y) -\Pi_j(\om, x)$, $\Pi_j$ satisfies the cocycle property.  Moreover,
\begin{align}
  \Norm{\Pi_j}{L_{\mathrm{cov}}^2}^2 
  \;=\;
  \mean\!\Big[{\textstyle \sum_x}\; \om_{0x} |x_j|^2\Big]
  \;=\;
  2 \mean[\om_{0e_j}]
  \;<\;
  \infty, 
\end{align}
where $e_j$ denotes the $j$th coordinate unit vector.  Hence, $\Pi^j \in L_{\mathrm{cov}}^2$.  So, we can define $\chi_j \in L_{\mathrm{pot}}^2$ and $\Phi_j \in L_{\mathrm{sol}}^2$ by the property
\begin{align*}
  \Pi_j
  \;=\;
  \chi_j \,+\, \Phi_j
  \;\in\;
  L_{\mathrm{pot}}^2 \oplus L_{\mathrm{sol}}^2.
\end{align*}
This gives a definition of the corrector $\chi = (\chi_1, \dots, \chi_d) : \Om \times \bbZ^d \to \bbR^d$.  Note that conventions about the sign of the corrector differ -- compare \cite{SS04} and \cite{Bi11}.  We set
\begin{align}\label{eq:def:M}
  M_t \;=\; \Phi(\om, X_t) \;=\; X_t - \chi(\om, X_t).
\end{align}
The following proposition summarizes the properties of $\chi$, $\Phi$ and $M$; see, for example, \cite{ABDH12}, \cite{BD10} or \cite{Bi11} for detailed proofs.
\begin{prop} \label{prop:constr_corr}
  Assume that $\prob$ satisfies the Assumption~\ref{ass:environment}. Then, for $\prob$-a.e.\ $\om$,
  \begin{align} \label{eq:Phi}
    \cL^\om \Phi(x)
    \;=\;
    \sum_{y\in \bbZ^d} \om_{xy} \big(\Phi(\om,y) - \Phi(\om,x) \big)
    \;=\;
    0,
    \qquad \Phi(\om,0) \;=\;0.
  \end{align}
  In other words, for $\prob$-a.e.\ $\om$ and for every $v \in \bbR^d$, $M$ and $v\cdot M$ are $\Prob^\om_0$-martingales.  The covariance process of the latter is given by
  \begin{align*}
    \langle v \cdot M \rangle_t
    \;=\;
    \int_0^t
      \sum_x (\tau_{X_s}\om )_{0x}\,
      \big(v \cdot \Phi(\tau_{X_s}\omega,x)\big)^2
    \md s.
  \end{align*}
  Moreover, provided that $\mean\!\big[1/\om(e)\big] < \infty$ for all $e \in E_d$, it holds that $\chi(\,\cdot\,,x) \in L^1(\prob)$ with $\mean[\chi(\om, x)] = 0$ for all $x \in \bbZ^d$.
\end{prop}
Based on the above construction, we now show that an invariance principle holds for $M$.  This is standard and follows from the ergodicity of the environment and the process of the \emph{environment as seen from the particle}.  We shall proceed as in \cite{ABDH12} and \cite{MP07}.

Recall that by the irreducibility of the random walk the ergodicity of the shift operator transfers to the process of the environment seen from the particle which is crucial in the proof of the invariance principle for the martingale part.  The environment seen from the particle is defined as the process $\{\tau_{X_t} \om : t \geq 0\}$, taking values in the environment space $\Om$, whose generator is given by
\begin{align*} 
  \widehat{\cL}\, \phi(\om)
  \;=\;
  \sum_{x \in \bbZ^d} \om_{0x}\, \big( \phi(\tau_x\om) - \phi(\om) \big)
\end{align*}
and the transition semigroup is given by
\begin{align*} 
  \widehat{\cP}_t \phi(\om)
  \;=\;
  \sum_{x \in \bbZ^d} p^\om(t,0,x)\, \phi(\tau_x\om),
  \qquad t \geq 0.
\end{align*}
\begin{lemma}\label{lem:eta}
  Suppose that Assumption \ref{ass:environment} holds.  Then the measure $\prob$ is stationary, reversible and ergodic for the environment process $\{\tau_{X_t}\om\}_{t \geq 0}$.
\end{lemma}
\begin{proof}
  By the invariance of $\prob$ with respect to space shifts, we have that
  \begin{align*}
    \mean [\widehat{\cL}\, \phi]
    &\;=\;
    \sum_{x\in \bbZ^d}
    \Big(
      \mean\!\big[(\tau_x \om)_{-x0}\, \phi\circ \tau_x\big]
      - \mean\!\big[\om_{0x}\, \phi\big]
    \Big)
    \nonumber\\
    &\;=\;
    \sum_{x\in \bbZ^d}
    \Big(
      \mean\big[\om_{0,-x}\, \phi\big] -  \mean\big[\om_{0x}\, \phi\big]
    \Big)
    \;=\;
    0.
  \end{align*}
  Thus, $\prob$ is an invariant measure for $\{\tau_{X_t}\om\}$.  To prove that $\prob$ is also ergodic, let now $A \in \cF$ with $\widehat{\cP}_t \indicator_A = \indicator_A$.  Then, for all $\om \in \Om$ and all $t \geq 0$,
  \begin{align*}
    0
    \;=\;
    \indicator_{A^c}(\om)\, \big(\widehat{\cP}_t \indicator_A\big)(\om)
    \;=\;
    \sum_{x \in \bbZ^d} \indicator_{A^c}(\om)\, p^\om(t,0,x)\,
    \indicator_A(\tau_{x}\om).
  \end{align*}
  In particular, $\indicator_{A^c}(\om)\, p^\om(t,0,x)\, \indicator_A(\tau_{x}\om)= 0$  for all $\om \in \Om$, all $t \geq 0$ and every $x \in \bbZ^d$.  But from Assumption \ref{ass:environment} we can deduce that the random walk $X$ is irreducible in the sense that for every $x \in \bbZ^d$
  \begin{align*}
    \prob\!%
    \Big[
      \big\{ \om : \, \sup\nolimits_{t\geq 0}\, p^\om(t, 0, x) \,>\, 0 \big\}
    \Big]
    \;=\;
    1.
  \end{align*}
  Hence, for every $x \in \bbZ^d$
  \begin{align*}
    \indicator_{A^c}(\om) \cdot \indicator_A(\tau_{x}\om)
    \;=\;
    0
    \qquad \text{for } \prob\text{-a.e. } \om.
  \end{align*}
  Thus, there exists a set $N$ with $\prob[N] = 0$ such that the set $A \setminus N$ is invariant under $\tau_{x}$. Since $\prob$ is ergodic with respect to $\tau_{x}$, we conclude that $A$ is $\prob$-trivial and the claim follows.
\end{proof}
\begin{prop} \label{prop:mconv}
  Suppose Assumption \ref{ass:environment} holds and assume that $\mean\!\big[1/\om(e)\big] < \infty$ for any $e \in E_d$.  Further, let $M^{(n)}_t \ldef \frac{1}{n} M_{n^2t}$, $t \geq 0$.  Then, for $\prob$-a.e.\ $\om$, the sequence of processes $\{M^{(n)}\}$ converges in law in the Skorohod topology to a Brownian motion with a non-degenerate covariance matrix $\Si^2$ given by
  \begin{align*}
  \Si_{ij}^2
  \;=\;
  \mean\!%
  \Big[
    {\textstyle \sum_{x \in \bbZ^d}}\; \om_{0x}\, \Phi_i(\om, x)\, \Phi_j(\om, x)
  \Big].
  \end{align*}
\end{prop}
\begin{proof}
  The proof is based on the martingale convergence theorem by Helland (see Theorem 5.1a) in \cite{He82}); see \cite{ABDH12} or \cite{MP07} for details. The argument is based on the fact that the quadratic variation of $M^{(n)}$ converges, for which the ergodicity of the environment process in Lemma~\ref{lem:eta} is needed.  Finally, we refer to Proposition 4.1 in \cite{Bi11} for a proof that $\Si^2$ is nondegenerate.
\end{proof}

\subsection{Sublinearity of the corrector} \label{sect:sublin}
To start with, let us denote by $B(x,r)$ a closed ball with respect to the graph distance with center $x \in \bbZ^d$ and radius $r$.  To lighten notation, we write $B(r) = B(0,r)$ and $\chi^{(n)}(\om, x) \ldef \frac{1}{n}\, \chi(\om, x)$.  The cardinality of $A \subset \bbZ^d$ is denoted by $|A|$.  Further, for a nonempty, finite $A \subset \bbZ^d$, we define a locally space-averaged norm on functions $f\!: \bbZ^d \to \bbR$ by
\begin{align*}
  \Norm{f}{p,A}
  \;\ldef\;
  \bigg(
    \frac{1}{|A|}\; \sum_{x \in A}\, |f(x)|^p
  \bigg)^{\!\!1/p}, \qquad p\in[1,\infty).
\end{align*}
The key ingredient in the proof of Theorem~\ref{thm:main} is the sublinearity of the corrector which we formalize as
\begin{prop} \label{prop:sublin_corr}
  Let $d \geq 2$ and suppose that Assumption \ref{ass:environment} and the moment condition \eqref{eq:moment_condition} hold.  Then, for any $L \geq 1$ and $j = 1, \ldots, d$,
  \begin{align} \label{eq:sublin_corr}
    \lim_{n \to \infty} \max_{x \in B(L n)} \big| \chi_j^{(n)}(\om, x) \big|
    \;=\;
    0,
    \qquad \prob\text{- a.s.}
  \end{align}
\end{prop}
The proof is based on both ergodic theory and purely analytic tools.  Using the spatial ergodic theorem, we show in a first step that the corrector $\chi^{(n)}$ averaged over cubes with side length of order $n$ vanishes $\prob$-a.s.\ when $n$ tends to infinity.  In a second step, we show by means of the $\ell^1$-Poincar\'e inequality on $\bbZ^d$ that $\chi^{(n)}$ also converges to zero in the $\Norm{\,\cdot\,}{1,B(n)}$-norm.  The final step uses the maximum inequality, which we establish in the next section, to bound from above the maximum of $\chi^{(n)}$ in $B(n)$ by $\Norm{\chi^{(n)}}{1,B(n)}$. 

We start with some immediate consequences from the ergodic theorem.  To simplify notation, let us define the following measures $\mu^{\om}$ and $\nu^{\om}$ on $\bbZ^d$:
\begin{align*}
  \mu^{\om}(x)
  \;=\;
  \sum_{x \sim y}\, \om_{xy}
  \qquad \text{and} \qquad
  \nu^{\om}(x)
  \;=\;
  \sum_{x \sim y}\, \frac{1}{\om_{xy}}.
\end{align*}
\begin{lemma} \label{lem:ergodic_const}
  Suppose $\mean\!\big[(\om(e))^p\big] < \infty$ and $\mean\!\big[(1/\om(e))^q\big] < \infty$ for some $p,q \in [1, \infty)$.  Then, for $\prob$-a.e.\ $\om$,
  \begin{align*}
    \lim_{n\to \infty} \Norm{\mu^{\om}}{p, B(n)}^p
    \;=\;
    \mean\!\big[ \mu^{\om}(0)^p\big]
    \qquad \text{and} \qquad
    \lim_{n\to \infty} \Norm{\nu^{\om}}{q, B(n)}^q
    \;=\;
    \mean\big[\nu^{\om}(0)^q\big]
  \end{align*}
  Further, if $\mean\!\big[1/\om(e)\big] < \infty$ then we have for every $j = 1, \ldots, d$,
  \begin{align} \label{eq:lim_form}
    \lim_{n\to \infty} 
    \frac{1}{|B(n)|}
    \sum_{\substack{x,y \in B(n)\\ x \sim y}}\mspace{-6mu}
    \big| \chi_j(\om,y) - \chi_j(\om,x) \big|
    \;\leq\;
    \mean\!\big[\nu^{\om}(0)\big]^{1/2}\, \Norm{\chi_j}{L_{\mathrm{cov}}^2}
  \end{align}
\end{lemma}
\begin{proof}
  The first two assertions are an immediate consequence of the spatial ergodic theorem.  For instance, we have
  \begin{align*} 
    \lim_{n\to \infty} \Norm{\mu^{\om}}{p, B(n)}^p
    \;=\;
    \lim_{n\to \infty}\frac{1}{|B(n)|}
    \sum_{x \in B( n)}\! \big( \mu^{\tau _x \om}(0) \big)^p
    \;=\;
    \mean\!\big[\mu^{\om}(0)^p\big].
  \end{align*}
  To prove the last assertion, notice first that by the Cauchy--Schwarz inequality we have for every $j = 1, \ldots, d$,
  \begin{align*}
    \mean\!\Big[{\textstyle \sum_{0 \sim x}}\, |\chi_j(\om, x)|\Big]^2
    \;\leq\;
    \mean\!\big[\nu^{\om}(0)\big]\;
    \mean\!\Big[
      {\textstyle \sum_{0 \sim x}}\, \om_{0x} |\chi_j(\om,x)|^2
    \Big]
    \;=\;
    \mean\!\big[\nu^{\om}(0)\big]\, \Norm{\chi_j}{L_{\mathrm{cov}}^2}^2.
  \end{align*}
  Due to the fact that $\mean\!\big[1/\om(e)\big] < \infty$ for all $e \in E_d$ and $\chi_j \in L_{\mathrm{cov}}^2$ the right-hand side of the equation above is finite.  Moreover, by the cocycle property it holds that $\chi_j(\om,y) - \chi_j(\om,x) = \chi_j(\tau_x \om, y-x)$.  Thus, a further application of the spatial ergodic theorem yields
  \begin{align*} 
    \lim_{n \to \infty}  \frac{1}{|B(n)|}
    \sum_{\substack{x,y \in B(n)\\ x \sim y}}\mspace{-12mu}
    \big| \chi_j(\tau_x \om, y-x) \big|
    \;\leq\;
    \mean\!\Big[{\textstyle \sum_{0 \sim x}}\, |\chi_j(\om,x)|\Big]
    \;\leq\;
    \mean\!\big[\nu^{\om}(0)\big]^{1/2} \Norm{\chi_j}{L_{\mathrm{cov}}^2}.
  \end{align*}
\end{proof}
\begin{lemma} \label{lem:cubes}
  Suppose Assumption \ref{ass:environment} holds and assume that $\mean\!\big[1/\om(e)\big] < \infty$.  Let $C$ be any cube in $\bbR^d$ of the form $C = \prod_{i=1}^d [a_i,b_i]$  with $a_i < b_i$, $i = 1, \ldots, d$ and set $C(n) \ldef n C \cap \bbZ^d$.  Then, for any $j = 1, \ldots, d$ and $\prob$-a.e.\ $\om$,
  \begin{align}\label{eq:cubes}
    \lim_{n \to \infty} \frac{1}{n^d} \sum_{x\in C(n)} \chi_j^{(n)}(\om,x)
    \;=\;
    0.
  \end{align}
\end{lemma}
\begin{proof}
  We will restrict the proof to the case where $C$ is of the form $C = \prod_{i=1}^d [0,L_i]$ with $L_i > 0$, $i = 1, \ldots, d$.  For general $C$, the statement follows by similar arguments.  We will proceed as in \cite[pp.\ 229--230]{SS04}.  Let us denote by $C^j(n) \ldef \prod_{i=1}^j [0, nL_i ] \times \{0\}^{d-j}$, $j = 1, \ldots, d$.  When $x = (x_1,\ldots, x_d) \in \bbZ^d$, we write $x = (y, x_d)$ with $y = (x_1, \ldots, x_{d-1}) \in \bbZ^{d-1}$, and we identify $\bbZ^{d-1}$ with $\bbZ^{d-1} \times \{0\} \subseteq \bbZ^d$.  Then, by Lemma~\ref{basicG}, we have for $\prob$-a.e.\ $\om$ and for any $j = 1, \ldots, d$,
  \begin{align*}
    \frac{1}{n^d}\,&\sum_{x \in C^d(n)} \chi_j^{(n)}(\om,x)
    \\
    &\;=\;
    \frac{1}{n^d} \sum_{\substack{y \in C^{d-1}(n)\\ 0\leq x_d \leq  n L_d}}
    \Bigg(
      \chi_j^{(n)}(\omega,y) \,+\!
      \sum_{k=0}^{x_d  - 1}\! \chi_j^{(n)}\big(\tau_{y + k e_d} \om, e_d\big)
    \Bigg)
    \\
    &\;=\;
    \frac{[nL_d]+1}{n} \frac{1}{n^{d-1}} \sum_{y \in C^{d-1}(n)}
    \chi_j^{(n)}(\om, y)
    \,+\,
    \frac{1}{n^d} \sum_{x \in C^d(n)} \mspace{-6mu} \frac{[n L_d] - x_d}{n}\,
    \chi_j(\tau_x \om, e_d).
  \end{align*}
  (Here $[\cdot]$  denotes the integer part.)  Since $\mean[|\chi_j(\om, x)|] < \infty$ for all $x \in \bbZ^d$, an application of the spatial ergodic theorem (see Theorem 3 in \cite{BD03}) gives that $\prob$-a.s.\ and in $L^1(\prob)$
  \begin{align*}
    \lim_{n\to \infty} \frac{1}{n^d} \sum_{x \in C^d(n)} \frac{[nL_d]-x_d}{n}\,
    \chi_j(\tau_x \om, e_d)
    \;=\;
    \int_C \big(L_d - v_d\big)\, \md v \, \mean\!\big[\chi_j(\om, e_d)\big]
    \;=\;
    0.
  \end{align*}
  The claim follows now by induction.  Indeed, in each step we use the spatial ergodic theorem with respect to the subgroup of space shifts to obtain that the limit
  \begin{align*}
    \lim_{n\to \infty} \frac 1 {n^{d-k}}
    \sum_{x \in C^{d-k}(n)} \mspace{-12mu} \frac{[n L_{d-k}] - x_{d-k}}{n}\,
    \chi_j^{(n)}(\tau_x\om, e_{d-k})
    \;\rdef\;
    F_k(\om),
  \end{align*}
  exists $\prob$-a.s.\ and in $L^1(\prob)$ for every $k = 0, \ldots, d-1$.  Further, $\mean[F_k] = 0$ and by construction it is clear that $F_k \circ \tau_{e_1} = F_k\;$ $\prob$-a.s.\ for every $k$.  Hence,
  \begin{align*}
    \lim_{n\to\infty} \frac{1}{n^d}\, \sum_{x \in C(n)} \chi_j^{(n)}(\om, x)
    \;\rdef\;
    F(\om),
  \end{align*}
  also exists $\prob$-a.s.\ and in $L^1(\prob)$ and is invariant under $\tau_{e_1}$.  By symmetry in the above calculation, $F$ is also invariant under $\tau_{e_i} $, $i = 1, \ldots, d$. Therefore, by using ergodicity we have that $F = 0$.
\end{proof}

\begin{prop}\label{prop:conv_l1_corr}
  Suppose Assumption \ref{ass:environment} holds and assume that $\mean\!\big[1/\om(e)\big] < \infty$.  Then, for any $j = 1, \ldots, d$ and $\prob$-a.e.\ $\om$, 
  \begin{align}\label{eq:l1:conv}
    \lim_{n\to \infty}\, \frac{1}{n^d}\,
    \sum_{x \in B(n)} \big|\chi_j^{(n)}(\om, x)\big|
    \;=\;
    0.
  \end{align}
\end{prop}
\begin{proof}
  The following proof is based on an argument similar to the one given in \cite{MP07}.  For any $k \in \bbN$, consider a partition of the cube $[-1,1]^d$ into $k^d$ cubes $C_i = C_{i,k}$, $i = 1, \ldots k^d$, with side length $1/k$.  For $n \geq 2 k$, set $C_i(n) = C_{i,k}(n) \ldef n C_{i,k} \cap \bbZ^d$.  By construction, $B(n) = B(0,n)$ is contained in the union of the cubes $C_{i}(n)$.  Denoting by $z_i \in \bbZ^d$ the lattice point approximation of the barycenter of $C_i(n)$, we further have that $B_i(n) = B(z_i, 2 n / k ) \supset C_{i,k}(n)$.
  
  Consider a function $u\!: \bbZ^d \to \bbR$.  Then, by means of the $\ell^1$-Poincar\'{e} inequality on $\bbZ^d$ (see \cite{DS91}), we have that
  \begin{align*}
    \sum_{x \in B_i(n)}\mspace{-6mu} \big|u(x)\big|
    &\;\leq\;
    \bigg(1 + \frac{|B_i(n)|}{|C_i(n)|}\bigg)\,
    \sum_{x \in B_i(n)}\mspace{-6mu} \big|u(x) - u_{B_i(n)}\big|
    \,+\, |B_i(n)|\, \big|u_{C_i(n)}\big|
    \nonumber\\
    &\;\leq\;
    C_{\mathrm{P}}\,\frac{2n}{k}
    \bigg(1 + \frac{|B_i(n)|}{|C_i(n)|}\bigg)\,
    \sum_{\substack{x,y \in B_i(n)\\ x \sim y}}\mspace{-12mu} \big|u(x) - u(y)\big|
    \,+\, |B_i(n)|\, \big|u_{C_i(n)}\big|
  \end{align*}
  where $u_B = |B|^{-1} \sum_{x \in B} u(x)$ and $C_{\mathrm{P}} \in (0, \infty)$.  Mind that the ratio $|B_i(n)| / |C_i(n)|$ is bounded from above by a constant independent of $k$.  Moreover, any edge $\{x,y\}$ with $x,y \in B(2n)$ is contained in at most $2^d$ different balls $B_i(n)$.  Thus, by summing over $i$ we obtain
  \begin{align}\label{eq:l1:cor}
   \sum_{x \in B(n)}\mspace{-6mu} \big|u(x)\big|
   \;\leq\;
   c\, \frac{n}{k}\,
   \sum_{\substack{x,y \in B(2 n)\\ x \sim y}}\mspace{-12mu} \big|u(x) - u(y)\big|
   \,+ \sum_{1 \leq i \leq k^d}\mspace{-6mu} |B_i(n)|\, \big|u_{C_i(n)}\big|.
  \end{align}
  Let us now apply \eqref{eq:l1:cor} to $\chi_j^{(n)}(\om, x)$ with $j = 1, \ldots, d$.  Then, in view of \eqref{eq:cubes} and \eqref{eq:lim_form}, we obtain that for $\prob$-a.a.\ $\om$
  \begin{align*}
    &\limsup_{n \to \infty} \frac{1}{n^d}
    \sum_{x \in B(n)}\mspace{-6mu} \big|\chi_j^{(n)}(\om, x)\big|
    \nonumber\\
    &\mspace{36mu}\leq\;
    \frac{c}{k}\, \mean\!\big[\nu^{\om}(0)\big]^{1/2}\;
    \Norm{\chi_j}{L_\mathrm{cov}^2}
    \,+\,
    \sum_{1 \leq i \leq k^d} \frac{c}{k^d}\,
    \bigg|
      \limsup_{n \to \infty}\,
      \frac{1}{(n/k)^d} \sum_{x \in C_{i,k}(n)}\mspace{-12mu} \chi_j^{(n)}(\om, x)
    \bigg|
    \nonumber\\
    &\mspace{36mu}\leq\;
    \frac{c}{k}\, \mean\!\big[\nu^{\om}(0)\big]^{1/2}\;
    \Norm{\chi_j}{L_\mathrm{cov}^2},
  \end{align*}
  and since $k$ is arbitrary, the claim follows.
\end{proof}
\begin{remark}
  Under the assumption that $\mean\!\big[(1/\om(e))^{p/(2-p)}\big] < \infty$, one can prove by arguments similar to the ones given in the proof of Lemma~\ref{lem:cubes} that $\Norm{\chi^{(n)}}{p,B(n)} < \infty$ (cf.\ also \cite{SS04}).  Combined with the so-called sublinearity on average of the corrector, proven in Proposition 4.15 in \cite{Bi11} (cf.\ also \cite{BB07}), this also allows to deduce \eqref{eq:l1:conv}, but under stronger conditions on the inverse moments.
\end{remark}
\begin{remark}
  Note that on $\bbZ^d$ the $\ell^p$-Poincar\'{e} inequality for $p > d$ also holds.  As it was shown in \cite[Th\'{e}or\`{e}me~4.1]{Cou96} (mind the typo in the statement), the $\ell^p$-Poincar\'{e} inequality implies a version of the Gagliardo--Nirenberg inequality.  By using this inequality instead of the $\ell^1$-Poincar\'{e} inequality in the proof of Proposition~\ref{prop:conv_l1_corr}, one can show the following estimate:
  \begin{align*}
    \max_{x \in B(n)} \frac{1}{n}\, \big|\chi_j(\om, x)\big|
    \;\leq\;
    \frac{c}{k}
    \bigg( \frac{1}{n^d}\!
      \sum_{\substack{x \in B(2 n)\\ x \sim y}}\mspace{-9mu}
      \big|\chi_j(\tau_x \om, y-x)\big|^p
    \bigg)^{\!1/p}
    + \sum_{1 \leq i \leq k^d} \frac{c}{k^d}\,
    \Big|\Big(\tfrac{1}{n}\chi_j\Big)_{C_i(n)}\Big|.
  \end{align*}
  Thus, provided that $\mean\!\big[|\chi(\om, x)|^p\big] < \infty$ for $p > d$ -- which can be established in certain well mixing situations \cite{GO11} -- the sublinearity of the corrector is immediate (cf.\ \cite{BD91}).
\end{remark}
The next proposition will be proven in a more general context in Section~\ref{sec:mos_it} below.
\begin{prop} \label{lem:mos_it}
  Let $p, q \in [1, \infty)$ be such that $1/p + 1/q < 2/d$.  Then, for every $\al > 0$, there exist $\ga' > 0$ and $\ka' > 0$ and $c \equiv c(p, q, d) < \infty$ such that
  \begin{align} \label{eq:mos_it}
    \max_{x \in B(n)} \big| \chi_j^{(n)}(\om, x) \big|
    \;\leq\;
    c\,
    \Big( 1 \vee
      \Norm{\mu^{\om}}{p, B(2 n)}\, \Norm{\nu^{\om}}{q,B(2 n)}
    \Big)^{\!\ka}\,
    \Norm{\chi_j^{(n)}(\om, \cdot)}{\al, B(2 n)}^{\ga'}
  \end{align}
  for $j = 1, \ldots, d$.
\end{prop}
\begin{proof}
  It is obvious that $\bbZ^d$ satisfies the properties of the general graphs considered in Section \ref{sec:mos_it}.  Then the assertion for $\chi_j^{(n)}$ follows directly from Corollary~\ref{cor:MP:general} with $\si = 1$, $\si' = 1/2$ and $n$ replaced by $2n$. Note that, in view of \eqref{eq:Phi} and \eqref{eq:def:M}, the function $V_j^{\om}$ appearing in Corollary~\ref{cor:MP:general} is given $V_j^{\om}(x,y) = \frac{1}{n}\, \om_{xy}\, (y_j - x_j)$.
\end{proof}

Proposition~\ref{prop:sublin_corr} is now immediate from Proposition~\ref{lem:mos_it} with the choice $\alpha=1$, Proposition~\ref{prop:conv_l1_corr} and Lemma~\ref{lem:ergodic_const}.

\subsection{Proof of Theorem \ref{thm:main}}
In order to conclude the proof of the invariance principle, it remains to show an almost sure uniform control of the corrector, which is a direct consequence from the sublinearity of corrector established in Section \ref{sect:sublin}.
\begin{prop} \label{prop:contr_corr}
  Let $T > 0$. For $\prob$-a.e.\ $\om$,
  \begin{align}
    \sup_{0 \,\leq\, t \,\leq\, T}\,
    \frac{1}{n}\, \Big| \chi\big(\om, n\, X_{t}^{(n)}\big) \Big|
    \;\underset{n \to \infty}{\longrightarrow}\;
    0 
    \quad \text{ in $\Prob_{\!0}^\om$-probability}.
  \end{align}
\end{prop}
\begin{proof}
  We proceed as in \cite{FK97,FK99}.  Recall that $\chi^{(n)}(\om, x) \ldef \frac{1}{n}\, \chi(\om, x)$.  Fix $T > 0$, $L > 1$ and denote by $T_{L,n}$ the exit time of $X^{(n)}$ from the cube $C=[-L,L]^d$.  By Proposition~\ref{prop:sublin_corr}, we have
  \begin{align*}
    \lim_{n \to \infty}\;
    \sup_{0 \,\leq\, t \,<\, T_{L,n}}\;
    \Big| \chi^{(n)}\big(\om, n\, X_{t}^{(n)}\big) \Big|
    \;=\;
    0,
    \qquad \prob\text{-a.s.}
  \end{align*}
  Hence, we can choose $n_0 \in \bbN$ such that $\sup_{0 \,\leq\, t \,<\, T_{L,n}} \big| \chi^{(n)}\big(\om, n\, X_{t}^{(n)}\big) \big| < 1$ for all $n \geq n_0$.  Then, for such $n$ we have
  \begin{align*}
    \Prob_{\!0}^{\om}\!%
    \bigg[ \sup_{0 \,\leq\, t \,\leq\, T} \big| X^{(n)}_t\big| > L \bigg]
    &\;=\;
    \Prob_{\!0}^{\om}\!%
    \bigg[
      T_{L,n} \leq T,
      \sup_{0 \,\leq\, t \,<\, T_{L,n}} \!\!\big| M^{(n)}_t\big| > L-1
    \bigg]
    \nonumber\\[.5ex]
    &\;\leq\;
    \Prob_{\!0}^{\om}\!%
    \bigg[ \sup_{0 \,\leq\, t \,\leq\, T} \big| M^{(n)}_t\big| \geq L-1 \bigg].
  \end{align*}
  Since $M^{(n)}$ converges weakly to a Brownian motion, we have by Doob's maximal inequality that there exists $c < \infty$ such that 
  \begin{align*}
    \limsup_{n \to \infty}\; \Prob_{\!0}^{\om}\!%
    \bigg[ \sup_{0 \,\leq\, t \,\leq\, T} \big| M^{(n)}_t\big| \geq L-1 \bigg]
    \;\leq\;
    \frac{c}{L-1}.
  \end{align*}
  Thus, for any $\de > 0$, we have
  \begin{align}
    &\limsup_{n\to \infty}\,
    \Prob_{\!0}^{\om}\!%
    \bigg[
      \sup_{0 \,\leq\, t \,\leq\, T}\,
      \Big| \chi^{(n)}\big(\om, n\, X_{t}^{(n)}\big) \Big| \geq \de
    \bigg]
    \nonumber\\[.5ex]
    &\mspace{36mu}\leq\;
    \limsup_{n \to \infty}
    \Bigg(
      \Prob_{\!0}^{\om}\!%
      \bigg[
        \sup_{0 \,\leq\, t \,<\, T_{L,n}}
        \Big| \chi^{(n)}\big(\om, n\, X_{t}^{(n)}\big) \Big|
        \geq \de
      \bigg]
      \,+\,
      \Prob_{\!0}^{\om}\!
      \bigg[ \sup_{0 \,\leq\, t \,\leq\, T} \big| X^{(n)}_t \big| \geq L \bigg]\,
    \Bigg)
    \nonumber\\[.5ex]
    &\mspace{36mu}\leq\;
    \frac{c}{L-1}.
  \end{align}
  Since $L>1$ is arbitrary, the claim follows.
\end{proof}
Theorem \ref{thm:main} now follows from Proposition \ref{prop:mconv} and Proposition \ref{prop:contr_corr}.

\section{Moser iteration on general weighted graphs}
\label{sec:mos_it}
Let us consider an infinite, connected, locally finite graph $G = (V, E)$ with vertex set $V$ and edge set $E$.  We will write $x \sim y$ if $\{x,y\} \in E$.  A path of length $n$ between $x$ and $y$ in $G$ is a sequence $\{x_i : i = 0, \ldots , n\}$ with the property that $x_0 = x$, $x_n = y$ and $x_i \sim x_{i+1}$.  Let $d$ be the natural graph distance on $G$, that is, $d(x,y)$ is the minimal length of a path between $x$ and $y$.  We denote by $B(x,r)$ the closed ball with center $x$ and radius $r$, that is, $B(x,r) \ldef \{y \in V \mid d(x,y) \leq r\}$.

The graph is endowed with the counting measure, that is, the measure of $A \subset V$ is simply the number $|A|$ of elements in $A$.  For functions $f\!:A \to \bbR$, where either $A \subseteq V$ or $A \subseteq E$, the $\ell^p$-norm $\norm{f}{p}{A}$ will be taken with respect to the counting measure.  The corresponding scalar products in $\ell^2(V)$ and $\ell^2(E)$ are denoted by $\scpr{\cdot}{\cdot}{V}$ and $\scpr{\cdot}{\cdot}{E}$, respectively.

For a given set $B \subset V$, we define the \emph{relative} internal boundary of $A \subset B$ by
\begin{align*}
  \partial_B A
  \;\ldef\;
  \big\{
    x \in A \;\big|\;
    \exists\, y \in B \setminus A\; \text{ s.th. }\; \{x,y\} \in E
  \big\}
\end{align*}
and we simply write $\partial A$ instead of $\partial_V A$.
\begin{assumption}\label{ass:graph}
 For some $d\geq 2$, the graph $G$ satisfies the following conditions:
  \begin{enumerate}[(i)]
    \item volume regularity of order $d$, that is, there exists $C_{\mathrm{reg}} \in(0, \infty)$ such that
      \begin{align}\label{eq:ass:vd}
       C^{-1}_{\mathrm{reg}}\, r^d  \;\leq\; |B(x,r)| \;\leq\; C_{\mathrm{reg}}\, r^d
        \qquad \forall\, x \in V,\; r \geq 1. 
      \end{align}
    \item relative isoperimetric inequality, that is, there exists $C_{\mathrm{riso}} \in (0, \infty)$ such that for all $x \in V$ and $r \geq 1$
      \begin{align}\label{eq:ass:riso}
        \frac{|\partial_{B(x,r)} A|}{|A|}
        \;\geq\;
        \frac{C_{\mathrm{riso}}}{r}
        \qquad
        \forall\; A \subset B(x,r)\; \text{ s.th. } |A| < \tfrac{1}{2} |B(x,r)|.
      \end{align}
  \end{enumerate}
\end{assumption}
\begin{remark}
  The Euclidean lattice, $(\bbZ^d, E_d)$, satisfies Assumption \ref{ass:graph}.
\end{remark}
\begin{remark} \label{rem:S_1}
  The following Sobolev inequality $(S_d^1)$ holds, that is,
  \begin{align}\label{eq:sob}
    \bigg(
      \sum\nolimits_{x \in V} |u(x)|^{d/(d-1)}
    \bigg)^{\!\!(d-1)/d}
    \;\leq\;
    C_{\mathrm{S_1}}
    \sum_{\{x,y\} \in E} \big| u(x) - u(y) \big|
  \end{align}
  for all functions $u\!: V \to \bbR$ with finite support.

  This can be seen as follows.  First, suppose that the graph $(V,E)$ satisfies condition (ii) in Assumption~\ref{ass:graph}.  Then, by means of a discrete version of the co-area formula, the classical $\ell^1$-Poincar\'e inequality can be easily established; see, for example, \cite[Lemma~3.3.3]{SC97}.  Second, provided that the counting measure also satisfies the doubling property and balls in $V$ have a regular volume growth, which are both ensured by the condition (i) in Assumption \ref{ass:graph}, the $\ell^1$-Poincar\'e inequality implies the usual isoperimetric inequality; see, for example, \cite[Proposition.~2.9]{Cou03}.  But, the latter is equivalent to the Sobolev inequality $(S_d^1)$, \cite[Proposition.~2.3]{Cou03}. 
\end{remark}
\begin{remark}
  It is well known \cite{BM03, MR04} that on random graphs, for example, on supercritical percolation clusters, the inequality \eqref{eq:ass:riso} holds only on large sets.  However, one can expect that for some $\th \in (0,1)$ and $x_0 \in V$ the following \emph{relative $\th$-isoperimetric inequality} holds:  There exists $R_0 = R_0(x_0)< \infty$ such that for all $R \geq R_0$ \eqref{eq:ass:riso} holds for every $B(x,r) \subset B(x_0,R)$ with $r > R^{\th}$.  Nevertheless, as it was communicated to us by M. Barlow, for functions $u$ with $\supp u \subset B(x_0,R)$ one can establish a Sobolev-type ineqality \eqref{eq:sob} provided that $d$ is replaced by the $d' = d/\ga$ with $\ga \in [0,1-\th)$ and $C_{\mathrm{S_1}}$ is replaced by $C_{\mathrm{S_1}} R^{1-\ga}$.  Let us stress that since $\supp u \subset B(x_0,R)$, we obtain even in this situation that 
  \begin{align*}
    \Bigg(
      \frac{1}{|B(x_0,R)|}
      \sum_{x \in B(x_0,R)}\mspace{-14mu} |u(x)|^{d' / (d'-1)}
    \Bigg)^{\!\!(d'-1)/d'}
    \mspace{-6mu}\leq\;
    C_{\mathrm{S_1}}\, \frac{R}{|B(x_0,R)|}\,
    \sum_{\substack{x,y \,\in\, B(x_0,R)\\\{x,y\} \in E}}\mspace{-20mu}
    \big| u(x) - u(y) \big|.
  \end{align*}
 In the particular case of supercritical percolation clusters on $\mathbb{Z}^d$, a stronger isoperimetric inequality has been proven in \cite{MR04}. More precisely, on the intersection of the cluster and a box with side length $n$ an isoperimetric inequality with dimension $d'(n)>d$ is obtained,  from which also such a type of Sobolev inequality follows.  
\end{remark}
Assume that the graph, $G$, is endowed with positive weights, that is, we consider a family $\om = \{\om(e) : e \in E\} \,\in\, (0, \infty)^E$.  To lighten notation, we set
$  \om(x,y) =\om(y,x) \ldef\om(\{x,y\})$ for all $ \{x,y\} \in E$ and $ \om(x,y) \ldef 0$ for all $ \{x,y\} \not\in E$.
Let us further define measures $\mu^{\om}$ and $\nu^{\om}$ on $V$ by
\begin{align*}
  \mu^{\om}(x) \;\ldef\; \sum_{x \sim y}\, \om(x,y)
  \qquad \text{and} \qquad
  \nu^{\om}(x) \;\ldef\; \sum_{x \sim y}\, \frac{1}{\om(x,y)}.
\end{align*}
For each non-oriented edge $e \in E$, we specify out of its two endpoints one as its initial vertex $e^-$ and the other one as its terminal vertex $e^+$. Nothing of what will follow depend on the particular choice.  Given a weighted graph $(V, E, \om)$, we define the \emph{discrete Laplacian}, $\cL^{\om}$, acting on bounded functions $f\!: V \to \bbR$ by
\begin{align*}
  \big(\cL^{\om} f\big)(x)
  \;\ldef\;
  \sum_{x \sim y}\, \om(x,y)\, \big(f(y) - f(x)\big)
  \;=\;
  - \nabla^*(\om \nabla f) (x),
\end{align*}
where the operators $\nabla$ and $\nabla^*$ are defined by $\nabla f\!: E \to \bbR$ and $\nabla^*F\!: V \to \bbR$
\begin{align*}
  \nabla f(e) \;\ldef\; f(e^+) - f(e^-),
  \qquad \text{and} \qquad
  \nabla^*F (x) \;\ldef\; \sum_{e: e^+ =\, x}\! F(e) \,-\! \sum_{e: e^- =\, x}\! F(e)
\end{align*}
for $f\!: V \to \bbR$ and $F\!: E \to \bbR$.  Mind that $\nabla^*$ is the adjoint of $\nabla$, that is, for all $f \in \ell^2(V)$ and $F \in \ell^2(E)$, it holds $\scpr{\nabla f}{F}{E} = \scpr{f}{\nabla^* F}{V}$.  We define the products $f \cdot F$ and $F \cdot f$ between a function, $f$, defined on the vertex set and a function, $F$, defined on the edge set in the following way:
\begin{align*}
  \big(f \cdot F\big)(e) \;\ldef\; f(e^-)\, F(e),
  \qquad \text{and} \qquad
  \big(F \cdot f\big)(e) \;\ldef\; f(e^+)\, F(e).
\end{align*}
Then the discrete analog of the product rule can be written as
\begin{align}\label{eq:rule:prod}
  \nabla(f g)
  \;=\;
  \big(g \cdot \nabla f\big) \,+\, \big(\nabla g \cdot f\big).
\end{align}
In contrast to the continuum setting, a discrete version of the chain rule cannot be established.  However, by means of the estimate \eqref{eq:A1:chain:ub1}, $|\nabla f^{\al}|$ for $f \geq 0$ can be bounded from above by
\begin{align}
  \label{eq:rule:chain1}
  \frac{1}{1 \vee |\al|}\,\big|\nabla f^{\al}\big|
  &\;\leq\;
  \big|f^{\al-1} \cdot \nabla f\big| \,+\, \big|\nabla f \cdot f^{\al-1}\big|,
  \qquad \forall\, \al \in \bbR.
  \intertext{On the other hand, the estimate \eqref{eq:A1:chain:lo} implies the following lower bound}
  \label{eq:rule:chain2}
  2\,\big|\nabla f^{\al}\big|
  &\;\geq\;
  \big|f^{\al-1} \cdot \nabla f\big| \,+\, \big|\nabla f \cdot f^{\al-1}\big|,
  \qquad \forall\, \al \geq 1.
\end{align}

The \emph{Dirichlet form} or \emph{energy} associated to $\cL^{\om}$ is defined by
\begin{align} \label{eq:ibp}
  \cE^{\om}(f,g)
  \;\ldef\;
  \scpr{f}{-\cL^{\om}g}{V}
  \;=\;
  \scpr{\nabla f}{\om \nabla g}{E},
  \qquad
  \cE^{\om}(f) \;\equiv\; \cE^{\om}(f,f).
\end{align}
For a given function $\eta\!: B \subset V \to \bbR$, we denote by $\cE_{\eta^2}^{\om}(u)$ the Dirichlet form where $\om(e)$ is replaced by $\frac{1}{2}(\eta^2(e^-) + \eta^2(e^+)) \om(e)$ for $e \in E$.  Mind that
\begin{align*}
  \cE_{\eta^2}^{\om}(u)
  \;=\;
  \scpr{\nabla u}{(\eta^2 \cdot \om)\, \nabla u}{E}.
\end{align*}

Finally, for any nonempty, finite $A \subset V$ and $p \in [1, \infty)$, we introduce space-averaged $\ell^{p}$-norms on functions $f\!: A \to \bbR$ by the usual formula
\begin{align*}
  \Norm{f}{p,A}
  \;\ldef\;
  \Bigg(
    \frac{1}{|A|}\; \sum_{x \in A}\, |f(x)|^p
  \Bigg)^{\!\! 1/p}
  \quad \text{and} \quad
  \Norm{f}{p, A, \mu^{\om}}
  \;\ldef\;
  \Bigg(
    \frac{1}{|A|}\; \sum_{x \in A}\,\mu^{\om}(x)\; |f(x)|^p 
  \Bigg)^{\!\! 1/p}\mspace{-12mu}.
\end{align*}

\subsection{Sobolev inequality}
The main objective in this subsection is to establish a weighted version of Sobolev inequality $(S_d^2)$ which is the key that allows us to use Moser's iteration technique.  Starting point for our further considerations is the Sobolev inequality $(S_d^1)$ on the unweighted graph $(V,E)$ in Remark~\ref{rem:S_1} 
\begin{align}\label{eq:sob:S1}
  \norm{u}{d/(d-1)}{V}
  \;\leq\;
  C_{\mathrm{S_1}}\, \norm{\nabla u}{1}{E}
\end{align}
for any function $u$ on $V$ with finite support.  Our task is to establish a corresponding version on a weighted graph.  For this purpose, define for $q \geq 1$
\begin{align}\label{eq:def:rho}
  \rho \;=\; \rho(d,q) \;\ldef\; \frac{d}{(d-2) + d/q}.
\end{align}
Notice that $\rho(d,q)$ is monotone increasing in $q$ and converges as $q$ tends to infinity to $d/(d-2)$. Moreover, $\rho(d, d/2) = 1$.
\begin{prop}[Sobolev inequality] \label{prop:sob}
  Suppose that the graph, $G = (V,E)$, satisfies Assumption~\ref{ass:graph} and let $B \subset V$ be finite and connected.  Consider a non-negative function $\eta$ with
  \begin{align*}
    \supp \eta \;\subset\; B, \qquad
    0 \;\leq\; \eta \;\leq\; 1 \qquad \text{and} \qquad
    \eta \equiv 0 \quad \text{on} \quad \partial B.
  \end{align*}
  Then, for any $q \in [1, \infty)$, there exists $C_{\mathrm{\,S}} \equiv C_{\mathrm{\,S}}(d,q) < \infty$ such that for any $u:V \to \bbR$,
  \begin{align}\label{eq:sob:ineq}
    \Norm{(\eta\, u)^2}{\rho,B}
    \;\leq\;
    C_{\mathrm{\,S}}\, |B|^{\frac{2}{d}}\, \Norm{\nu^{\om}}{q,B}\;
    \bigg(
      \frac{\cE_{\eta^2}^{\om}(u)}{|B|} \,+\,
      \norm{\nabla \eta}{\infty\!}{E}^2\,
      \Norm{u^2}{1,B, \mu^{\om}}
    \bigg).
  \end{align}
  If $d \geq 3$, \eqref{eq:sob:ineq} also holds for $q = \infty$.
\end{prop}
\begin{remark} \label{rem:sob}
  By H\"older's inequality with $1/p + 1/p_* = 1$ and $p \in (1, \infty]$, we have
  \begin{align*}
    \Norm{(\eta\, u)^2}{\rho / p_*, B, \mu^{\om}}
    \;\leq\;
    \Norm{\mu^{\om}}{p,B}^{p_* / \rho}\;
    \Norm{(\eta\, u)^2}{\rho, B, \mu^{\om}}.
  \end{align*}
  Thus, in view of \eqref{eq:sob:ineq}, we obtain that
  \begin{align}
    \Norm{(\eta\, u)^2}{r, B, \mu^{\om}}
    \;\leq\;
    C_{\mathrm{S}}\, |B|^{2/d}\,
    \Norm{\nu^{\om}}{q,B}\,
    \Norm{\mu^{\om}}{p,B}^{1/r}\;
    \bigg(
      \frac{\cE_{\eta^2}^{\om}(u)}{|B|} \,+\,
      \norm{\nabla \eta}{\infty\!}{E}^2\,
      \Norm{u^2}{1, B, \mu^{\om}}
    \bigg),
  \end{align}
  where $r \equiv r(d,p,q) \ldef \rho(d,q) / p_* = \big(d - d/p\big)/\big(d-2 + d/q\big)$.
\end{remark}
\begin{proof}
  First of all notice that $\nabla (\eta u) = \eta \cdot \nabla u + \nabla \eta \cdot u$ due to \eqref{eq:rule:prod}.  This implies that
  \begin{align} \label{eq:sob:prod}
    \cE^{\om}(\eta\, u)
    \;\leq\;
    2\, \cE_{\eta^2}^{\om}(u)
    \,+\,
    2\,
    \norm{\nabla \eta}{\infty\!}{E}^2\,
    \norm{u^2 \mu^{\om}}{1}{B}.
  \end{align}
  Hence, it suffices to prove that for any function $v\!: V \to \bbR$ with $\supp v \subset B$
   \begin{align}\label{eq:sob:ineq:noeta}
    \Norm{v^2}{\rho,B}
    \;\leq\;
    \frac{C_{\mathrm{\,S}}}{2}\, |B|^{2/d}\, \big\|\nu^{\om}\big\|_{q,B}\;
    \frac{\cE^{\om}(v)}{|B|}.
  \end{align}
  But, an application of \eqref{eq:sob:S1} to the function $|v|^{\al}$ with $\al = 2 \rho(q,d) (d-1)/d$ yields
  \begin{align}\label{eq:sob:d1:noeta}
    \norm{|v|^{\al}}{d/(d-1)}{V}
    &\;\leq\;
    C_{\mathrm{S_1}}\, \norm{\nabla |v|^{\al}}{1}{E}
    \nonumber\\
    &\;\leq\;
    2\, C_{\mathrm{S_1}}\, \max\{1,\al\}\, \norm{|v|^{\al-1} \cdot \nabla |v|}{1}{E},
  \end{align}
  where we used \eqref{eq:rule:chain1} in the last step.  By using the Cauchy--Schwarz inequality, we find
  \begin{align}\label{eq:sob:term1:noeta}
    \norm{|v|^{\al-1} \cdot \nabla |v|}{1}{E}
    \;\leq\;
    |B|^{1/2}\, \cE^{\om}(v)^{1/2}\, \Norm{|v|^{2(\al-1)}\,
      \nu^{\om}}{1,B}^{1/2},
  \end{align}
  where we used that $\cE^{\om}(|v|) \leq \cE^{\om}(v)$.  If $q = 1$, then $\al = 1$ and \eqref{eq:sob:ineq:noeta} follows immediately from \eqref{eq:sob:d1:noeta} and \eqref{eq:sob:term1:noeta} after normalizing the norms.  In the case $q > 1$, H\"{o}lder's inequality with $1/q + 1/q_* = 1$ and $q \in (1, \infty]$ yields
  \begin{align*}
    \Norm{|v|^{2(\al-1)}\, \nu^{\om}}{1,B}^{1/2}
    \;\leq\;
    \Norm{|v|^{2(\al-1)}}{q_*,B}^{1/2}\; \Norm{\nu^{\om}}{q,B}^{1/2}.
  \end{align*}
  But, due to the definition of $\rho(q,d)$, we have that $\al d/(d-1) = 2 q_* (\al-1) = 2 \rho$.  Thus, combining the last estimate with \eqref{eq:sob:d1:noeta} and \eqref{eq:sob:term1:noeta} and solving for $\Norm{v^2}{\rho,B}$, \eqref{eq:sob:ineq:noeta} is immediate.  
\end{proof}

\subsection{Maximum inequality for Poisson equations}
 In this section, our main objective is to establish a maximum inequality for the solution of a particular Poisson equation where the right-hand side is in divergence form.  More precisely, denote by $u$ the solution of 
\begin{align}\label{eq:poisson_eq}
  \cL^{\om} u \;=\; \nabla^* V^{\om}   ,
  \qquad \text{on} \quad B \subset V \text{ finite}
\end{align}
where $V^{\om}\!:E \to \bbR$ is given by
\begin{align}\label{eq:local:drift}
  V^{\om}(e) \ldef \om(e)\, \nabla f(e)
\end{align}
for some function $f\!:V \to \bbR$.
\begin{theorem} \label{thm:MP}
  For any $x_0 \in V$ and $n \geq 1$ let $B(n) \equiv B(x_0,n)$.  Suppose that $\cL^{\om} u = \nabla^* V^{\om}$ on $B(n)$.  Assume that the function $f$ in \eqref{eq:local:drift} satisfies $|\nabla f(e)| \leq 1/n$ for all $e \in E$.  Then, for any $p,q \in (1, \infty]$ with
  \begin{align}\label{cond:pq}
    \frac{1}{p} + \frac{1}{q} \;<\; \frac{2}{d} 
  \end{align}
  there exists $\ga \in (0,1]$, $\ka \equiv \ka(d,p,q) \in (1,\infty)$ and $C_1 \equiv C_1(d) < \infty$ such that for all $1/2 \leq \si' < \si \leq 1$
  \begin{align}\label{eq:MP}
    \max_{x \in B(\si' n)} |u(x)|
    \;\leq\;
    C_1\,
    \Bigg(
      \frac{1\vee \Norm{\mu^{\om}}{p,B(n)}\, \Norm{\nu^\om}{q,B(n)}}{(\si - \si')^2}
    \Bigg)^{\!\!\ka}\, \Norm{u}{2 \rho, B(\si n)}^{\ga},
  \end{align}
  where $\rho = \rho(q,d)$ is given by \eqref{eq:def:rho}.
\end{theorem}
As a first step, we establish the following lemma.
\begin{lemma}\label{lem:moser:DF}
  Let $B$ be a connected, finite subset of $V$ and $\eta$ be a nonnegative function with $\supp \eta \subset B$, bounded by $1$ and $\eta \equiv 0$ on $\partial B$.  Suppose that $\cL^{\om} u = \nabla^* V^{\om}$ on $B$.  Then there exists $C_2 < \infty$ such that for all $\al \geq 1$
  \begin{align}\label{eq:moser:DF}
    \frac{\cE_{\eta^2}^{\om}(\tilde{u}^{\al})}{|B|}
    &\;\leq\;
    C_2\, \tfrac{\al^4}{(2\al-1)^2}\;
    \norm{\nabla \eta}{\infty\!}{E}^2\;
    \Norm{u}{2\al, B, \mu^{\om}}^{2\al}
    \nonumber\\[1ex]
    &\mspace{32mu}+\,
    C_2\, \tfrac{\al^4}{(2\al-1)^2}\;
    \norm{\nabla f}{\infty\!}{E}^2\;
    \Norm{u}{2(\al-1),B, \mu^{\om}}^{2(\al-1)}
    \nonumber\\[1ex]
    &\mspace{32mu}+\,
    C_2\, \tfrac{\al^2}{2\al-1}\;
    \norm{(\nabla \eta)(\nabla f)}{\infty\!}{E}\;
    \Norm{u}{(2\al-1),B ,\mu^{\om}}^{2\al-1},
  \end{align}
  where $\tilde{u}^{\al}$ denotes the function $\tilde{u}^{\al} \ldef |u|^{\al} \cdot \sign u$.
\end{lemma}
\begin{proof}
  Since $u$ solves the Poisson equation $\cL^{\om} u = \nabla^* V^{\om}$, \eqref{eq:ibp} implies
  \begin{align*}
    \scpr{\nabla (\eta^2 \tilde{u}^{2\al-1})}{\om\, \nabla u}{E}
    \;=\;
    \scpr{\eta^2 \tilde{u}^{2\al-1}}{-\cL^{\om} u}{V}
    \;=\;
    - \scpr{\nabla (\eta^2 \tilde{u}^{2\al-1})}{V^{\om}}{E}.
  \end{align*}
  Hence,
  \begin{align}\label{eq:moser:weight:split1}
    &\scpr{\eta^2 \cdot \nabla \tilde{u}^{2\al-1}}{\om\, \nabla u}{E}
    \nonumber\\[.5ex]
    &\mspace{36mu}\leq\;
    \norm{\om\, (\nabla u)\,(\nabla \eta^2) \cdot |u|^{2\al-1}}{1\!}{E}
    \,+\,
    \norm{\nabla (\eta^2 \tilde{u}^{2\al-1})\, V^{\om}}{1\!}{E}.
  \end{align}
  As an immediate consequence of \eqref{eq:A1:pol:ub}, we get
  \begin{align} \label{eq:estDF}
    \scpr{\eta^2 \cdot \nabla \tilde{u}^{2\al-1}}{\om\, \nabla u}{E}
    \;\geq\;
    \tfrac{2\al-1}{\al^2}\; \cE_{\eta^2}^{\om}(\tilde{u}^{\al}).
  \end{align}
  The constant $c \in (0, \infty)$ appearing in the computations below is independent of $\alpha$ but may change from line to line.  Consider the first term on the right-hand side of \eqref{eq:moser:weight:split1}.  Since \eqref{eq:A1:chain:ub2} is applicable after a suitable symmetrization, we find
  \begin{align*}
    \norm{\om\, (\nabla u)\,(\nabla \eta^2) \cdot |u|^{2\al-1}}{1\!}{E}
    &\;\leq\;
    c\;
    \norm{\om\, (\nabla \tilde{u}^{\al})\,(\nabla \eta^2) \cdot |u|^{\al}}{1\!}{E}
    \nonumber\\[.5ex]
    &\;\leq\;
    c\, \ve\; \cE_{\eta^2}^{\om}(\tilde{u}^{\al})
    \,+\,
    \frac{c}{\ve}\, \norm{\nabla \eta}{\infty}{E}^{2}\,
    \norm{|u|^{2 \al} \mu^{\om}}{1}{B},
  \end{align*}
  where we used the Young's inequality $|a b| \leq \frac{1}{2}(\ve a^2 + b^2 / \ve)$.  On the other hand, by \eqref{eq:rule:prod} the second term on the right-hand side of \eqref{eq:moser:weight:split1} reads
  \begin{align*}
    \norm{\eta^2 \cdot (\nabla \tilde{u}^{2\al-1})\, V^{\om}}{1}{E}
    \,+\,
    \norm{V^{\om}\, (\nabla \eta^2) \cdot |u|^{2\al-1})}{1}{E}.
  \end{align*}
  Since $\eta$ is bounded by $1$,
  \begin{align*}
    \norm{V^{\om}\, \nabla \eta^2 \cdot |u|^{2\al-1}}{1}{E}
    \;\leq\;
    2\,
    \norm{(\nabla \eta)(\nabla f)}{\infty\!}{E}\,
    \norm{|u|^{2\al-1}\mu^{\om}}{1}{B}.
  \end{align*}
  By applying \eqref{eq:A1:chain:ub1} and again Young's inequality, we find
  \begin{align*}
    \norm{\eta^2 \cdot (\nabla \tilde{u}^{2\al-1})\, V^{\om}}{1}{E}
    &\;\leq\;
    c\;
    \norm{\eta^2\, |u|^{\al-1} \cdot \om\, (\nabla \tilde{u}^{\al})\, (\nabla f)}
         {1}{E}
    \nonumber\\[.5ex]
    &\;\leq\;
    c\, \ve\, \cE_{\eta^2}^{\om}(\tilde{u}^\al)
    \,+\,
    \frac{c}{\ve}\, \norm{\nabla f}{\infty}{E}^2\,
    \norm{|u|^{2(\al-1)} \mu^{\om}}{1}{B}
  \end{align*}
  Hence, by combining the estimates above and solving for $\cE_{\eta^2}^{\om}(\tilde{u}^\al)$, we obtain
  \begin{align*}
    \Big(
      1 - c\,\ve\, \tfrac{\al^2}{2\al-1}
    \Big)\;
    \frac{\cE_{\eta^2}^{\om}(\tilde{u}^{\al})}{|B|}
    \;\leq\;
    a + b
  \end{align*}
  with
  \begin{align*}
    a
    \;=\;
    \frac{c}{\ve}\; \tfrac{\al^2}{2\al-1}\,
    \Big(
      \norm{\nabla \eta}{\infty\!}{E}^2\,
      \Norm{|u|^{2 \al}}{1, B, \mu^{\om}}
      \,+\,
      \norm{\nabla f}{\infty\!}{E}^2\,
      \Norm{|u|^{2(\al-1)}}{1, B, \mu^{\om}}
    \Big)
  \end{align*}
  and
  \begin{align*}
    b
    &\;=\;
    c\; \tfrac{\al^2}{2 \al -1}\,\
    \norm{(\nabla \eta)(\nabla f)}{\infty\!}{E}\,
    \Norm{|u|^{2\al-1}}{1, B, \mu^{\om}}.
  \end{align*}
  By choosing $\ve = (2\al-1)/ 2 c\, \al^2$, the assertion follows.
\end{proof}
\begin{figure}[t]
  \label{fig:balls}
  \begin{center}
    \includegraphics{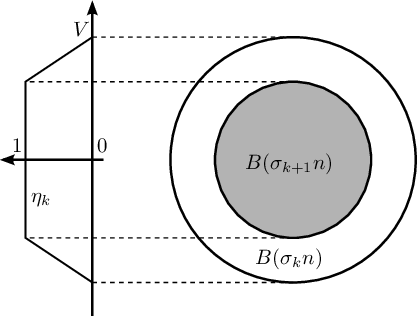}
  \end{center}
  \caption{Ilustration of both the balls $B(\si_k n)$ and $B(\si_{k+1} n)$ and the corresponding cut-off function $\eta_k$}
\end{figure}
\begin{proof}[Proof of Theorem~\ref{thm:MP}]
  For fixed $1/2 \leq \si' < \si \leq 1$, consider a sequence $\{B(\si_k n)\}_k$ of balls with radius $\si_k n$ centered at some $x_0$, where
  \begin{align*}
    \si_k \;=\; \si' + 2^{-k} (\si - \si')
    \qquad \text{and} \qquad
    \tau_k \;=\; 2^{-k-1} (\si - \si'),
    \quad k = 0, 1, \ldots
  \end{align*}
  Then $\si_{k} = \si_{k+1} + \tau_{k}$ and $\si_0 = \si$.  Further, for any $p, q \in (1,\infty)$ the condition in \eqref{cond:pq} implies that $\rho > p_* \equiv p/(p-1)$. Hence, by setting $\al_k = (\rho / p_*)^k$, we have that $\al_k > 1$ for every $k \geq 1$.   Due to the discrete structure of the underlying space $\bbZ^d$, the discrete balls $B(\si_{k+1} n)$ and $B(\si_k n)$ may coincide whenever $\tau_k n$ is sufficiently small.  For this reason, we proceed by distinguishing two different cases.  
  
  First consider the case $\tau_k n \geq 1$.  Let $\eta_k$ be a cut-off function with $\supp \eta_k \subset B(\si_{k} n)$ having the property that $\eta_k \equiv 1$ on $B(\si_{k+1} n)$, $\eta_k \equiv 0$ on $\partial B(\si_{k} n)$ and $\eta_k$ decays linearly on $B(\si_{k} n) \setminus B(\si_{k+1} n)$.  In particular, this choice of $\eta_k$ implies that $|\nabla \eta_k(e)| \leq 1/\tau_k n $ for all $e \in E$.  By applying the Sobolev inequality \eqref{eq:sob:ineq} to $\tilde{u}^{\al_k}$ and using H\"older's inequality, we obtain
  \begin{align*}
    &\Norm{(\eta\, \tilde{u}^{\al_k})^2}{\rho, B(\si_{k} n)}
    \nonumber\\[.5ex]
    &\mspace{36mu}\leq\;
    C_{\mathrm{S}}\, |B(\si_k n)|^{2/d}\, \Norm{\nu^\om}{q,B(\si_{k} n)}
    \Bigg(
      \frac{\cE_{\eta_k^2}^{\om}(\tilde{u}^{\al_k})}{|B(\si_{k} n)|}
      \,+\,
      \frac{\Norm{\mu^{\om}}{p,B(\si_k n)}}{(\tau_k n)^2}\,
      \,\Norm{\tilde{u}^{2\al_k}}{p_*,B(\si_{k} n)}
    \Bigg).
  \end{align*}
  On the other hand, by means of Jensen's inequality we obtain from \eqref{eq:moser:DF} that
  \begin{align*}
    \frac{\cE_{\eta_k^2}^{\om}(\tilde{u}^{\al_k})}{|B(\si_k n)|}
    \;\leq\;
    3\, C_2\, \Norm{\mu^{\om}}{p,B(\si_k n)}\,
    \bigg(\frac{\al_k}{\tau_k n}\bigg)^2\,
    \Norm{u}{2 \al_k p_*,B(\si_k n)}^{2 \al_k \ga_k},
  \end{align*}
  where $\ga_k = 1$ if $\|u\|_{2 \al_k p_*,B(\si_k n)} \geq 1$ and $\ga_k = 1 - \frac{1}{\al_k}$ if $\|u\|_{2 \al_k p_*,B(\si_k n)} < 1$.  By combining these two estimates and using that $\alpha_{k+1} p_* = \alpha_k \rho$, we find
  \begin{align}\label{eq:iteration:est1}
    \Norm{u}{2 \al_{k+1} p_*,B(\si_{k+1} n)}
    \;\leq\;
    \bigg(
      c\, \frac{2^{2k}\, \al_k^2}{(\si - \si')^{2}}\,
      \Norm{\mu^{\om}}{p,B(n)}\, \Norm{\nu^\om}{q,B(n)}
    \bigg)^{\!1/(2\al_k)}\,
    \Norm{u}{2 \al_k p_*,B(\si_k n)}^{\ga_k}
  \end{align}
  for some $c < \infty$.  Next we consider the case $\tau_k n < 1$. Note that
  \begin{align*}
    \Norm{|u|^{2 \al_k}}{\al p_*, B(\si_{k+1} n)}
    &\;\leq\;
    \Norm{|u|^{2 \al_k}}{p_*, B(\si_k n)}^{1/\al}
    \Big( \max_{x \in B(\si_k n)} |u(x)|^{2 \al_k} \Big)^{\!\!1/\al_*}
    \\[.5ex]
    &\;\leq\;
    |B(\si_k n)|^{1/(\al_* p_*)} \Norm{|u|^{2 \al_k}}{p_*, B(\si_k n)}.
  \end{align*}
  Since $d/(2 \al_* p_*) \leq 1$ and $n < 1 / \tau_k$, we find
  \begin{align}\label{eq:iteration:est2}
    \Norm{u}{2 \al_{k+1} p_*,B(\si_{k+1} n)}
    \;\leq\;
    \bigg(
      c\, \frac{2^{2k}}{(\si - \si')^{2}}\,
    \bigg)^{\!1/(2\al_k)}\,
    \Norm{u}{2 \al_k p_*,B(\si_k n)}^{\ga_k}.
  \end{align}
  Observe that $|B_{K}|^{1/2\al_K} \leq c$ uniformly in $n$ for any $K$ such that $\al_K \geq \ln n$. Hence,
  \begin{align*}
    \max_{x \in B(\si' n)} |u(x)|
    \;\leq\;
    |B(\si_K n)|^{1/(2 \al_K)}\, \Norm{u}{2 \al_K, B(\si_K n)}
    \;\leq\;
    c\, \Norm{u}{2 \al_n p_*, B(\si_K n)}
  \end{align*}
  By iterating the inequalities \eqref{eq:iteration:est1} and \eqref{eq:iteration:est2}, respectively, we get
  \begin{align*}
    \max_{x \in B(\si' n)} |u(x)|
    \;\leq\;
    C_1\, \prod_{k = 0}^{K-1}
    \Bigg(
      \frac{1 \vee \Norm{\mu^{\om}}{p,B(n)}\, \Norm{\nu^\om}{q,B(n)}}
      {(\si-\si')^2}
    \Bigg)^{\!\!1/(2\al_k)}\,
    \Norm{u}{2 \rho, B(\si n)}^{\ga}
  \end{align*}
  where $0 < \ga = \prod_{k=1}^{\infty} \ga_k  \leq 1$ and $C_1 < \infty$ is a constant independent of $k$, since $\sum_{k=0}^{\infty}k/\al_k < \infty$.  Finally, choosing $\ka = \frac{1}{2} \sum_{k=0}^\infty 1/\al_k <\infty$, the claim is immediate.
\end{proof}
\begin{corro}\label{cor:MP:general}
  Suppose $u$ satisfies the assumptions of Theorem~\ref{thm:MP}.  Then, for all $\al \in (0, \infty)$ and for any $1/2 \leq \si' < \si \leq 1$, there exists $C_3 < \infty$, $\ga' \equiv \ga'(\ga, \al, \rho)$ and $\ka' \equiv \ka'(\ka, \ga, \al, \rho) < \infty$ such that
  \begin{align} \label{eq:MP:general}
    \max_{x \in B(\si' n)} |u(x)|
    \;\leq\;
    C_3\,
    \Bigg(
      \frac{1\vee \Norm{\mu^{\om}}{p,B(n)}\, \Norm{\nu^{\om}}{q,B(n)}}
           {(\si - \si')^{2}}
    \Bigg)^{\!\!\ka'}\; 
    \Norm{u}{\al, B(\si n)}^{\ga'}.
  \end{align}
\end{corro}
\begin{proof}
  In view of \eqref{eq:MP}, for any $\al > 2 \rho$ the statement \eqref{eq:MP:general} is an immediate consequence of Jensen's inequality.  It remains to consider the case $\al \in (0, 2 \rho)$.  But from \eqref{eq:MP} we have for any $1/2 \leq \si' < \si \leq 1$
  \begin{align}\label{eq:MP:general1}
    \max_{x \in B(\si' n)} |u(x)|
    \;\leq\;
    C_2\,
    \Bigg(
      \frac{1\vee \Norm{\mu^{\om}}{p,B(n)}\, \Norm{\nu^{\om}}{q, B(n)}}
           {(\si - \si')^{2}}
    \Bigg)^{\!\!\ka}\;
    \Norm{u}{2\rho,B(\si n)}^{\ga}.
  \end{align}
  The remaining part of the proof follows the arguments in \cite[Theorem 2.2.3]{SC02}.  In the sequel, let $1/2 \leq \si' < \si \leq 1$ be arbitrary but fixed and set $\si_k = \si - 2^{-k}(\si - \si')$ for any $k \in \bbN_0$.  Now, by H\"older's inequality, we have for any $\al \in (0, 2\rho)$
  \begin{align*}
    \Norm{u}{2 \rho, B(\si_k n)}
    \;\leq\;
    \Norm{u}{\al, B(\si_k n)}^{\theta}\;
    \Norm{u}{\infty, B(\si_k n)}^{1-\th}
  \end{align*}
  where $\th = \al / 2 \rho$.  Hence, in view of \eqref{eq:MP:general1} and the volume regularity which implies that $|B(\si n)| / |B(\si' n)| \leq C_{\mathrm{reg}}^2 2^d$, it holds that
  \begin{align*}
    \Norm{u}{\infty, B(\si_{k-1} n)}
    \;\leq\;
    2^{2 \ka k}\, J\; \Norm{u}{\al, B(\si n)}^{\ga \th}\;
    \Norm{u}{\infty, B(\si_k n)}^{\ga - \ga \th},
  \end{align*}
  where we introduced $J = c\, \big(1 \vee \Norm{\mu^{\om}}{p,B(n)}\, \Norm{\nu^{\om}}{q, B(n)}/ (\si - \si')^{2} \big)^{\ka}$ to lighten notation.  By iteration, we get
  \begin{align*}
    \Norm{u}{\infty, B(\si' n)}
    \;\leq\;
    2^{2\ka \sum_{k=0}^{i-1} (k+1)(\ga - \ga \th)^k}\;
    \Big(
      J\; \Norm{u}{\al, B(\si n)}^{\ga \th}
    \Big)^{\sum_{k=0}^{i-1} (\ga - \ga\th)^k}\;
    \Norm{u}{\infty, B(\si_i n)}^{(\ga - \ga \th)^i}.
  \end{align*}
  Since $\ga(1 - \th) \in (0,1)$, as $i$ tends to infinity, this yields
  \begin{align*}
    \max_{x \in B(\si' n)} |u(x)|
    \;\leq\;
    2^{\frac{2\ka}{(1 - \ga + \ga\th)^2}}\, J^{\frac{1}{1 - \ga + \ga\th}}\;
    \Norm{u}{\al, B(\si n)}^{\frac{\ga \th}{1 - \ga + \ga \th}}
  \end{align*}
  and \eqref{eq:MP:general} is immediate.
\end{proof}

\appendix
\section{Technical estimates}
In this section, we provide proofs of some technical estimates needed in the proof of the Moser iteration.  In a sense, some of them may be seen as a replacement for a discrete chain rule.  Some extra care is required since the solution of the Poisson equation may be negative.
\begin{lemma}
  For $a\in \bbR$, we write $\tilde a^{\al}:=|a|^{\al} \cdot \sign a$ for any $\al \in \bbR \setminus \{0\}$.
  \begin{enumerate}[ (i) ]
    \item For all $a,b \in \bbR$ and any $\al, \be \ne 0$,
      \begin{align}\label{eq:A1:chain:ub1}
        \big| \tilde a^{\al} - \tilde b^{\al} \big|
        \;\leq\;
        \Big(1 \vee \Big|\frac{\al}{\be}\Big| \Big) \,
        \big|\tilde{a}^{\be} - \tilde{b}^{\be}\big|\,
        \big(\,|a|^{\al-\be} + |b|^{\al-\be} \big).
      \end{align}
    \item For all $a,b \in \bbR$ and any $\al > 1/2$,
      \begin{align}\label{eq:A1:pol:ub}
        \big(\tilde{a}^{\al} - \tilde{b}^{\al}\big)^2
        \;\leq\;
        \bigg|\frac{\al^2}{2\al-1}\bigg|\, \big(a - b\big)\,
        \big(\tilde{a}^{2\al-1} - \tilde{b}^{2\al-1}\big).
      \end{align}
      In particular, if $a,b \in \bbR_+$ then \eqref{eq:A1:pol:ub} holds for all $\al \not\in \{0,1/2\}$. 
    \item For all $a,b \in \bbR$ and any $\al,\be \geq 0$,
      \begin{align}\label{eq:A1:chain:lo}
        \big(|a|^{\al} + |b|^{\al}\big)\,
        \big|\tilde{a}^{\be} -\, \tilde{b}^{\be}\big|
        \;\leq\;
        2\, \big|\tilde{a}^{\al+\be} -\, \tilde{b}^{\al+\be}\big|.
      \end{align}
    \item For all $a,b \in \bbR$ and any $\al \geq 1/2$,
      \begin{align}\label{eq:A1:chain:ub2}
        \big(|a|^{2\al-1} + |b|^{2\al-1}\big)\, \big|a - b\big|
        \;\leq\; 
        4\,
        \big|\tilde a^{\al} - \tilde b^{\al}\big|\,
        \big(|a|^{\al} + |b|^{\al}\big).
      \end{align}
  \end{enumerate}
\end{lemma}
\begin{proof}
  \textit{(i)} First of all assume that $a,b \geq 0$ or $a,b \leq 0$.  Then, for all $\al, \be \ne 0$,
  \begin{align*}
    \big|\tilde{a}^{\al} - \tilde{b}^{\al}\big|
    \;=\;
    \big||a|^{\al} - |b|^{\al}\big|
    \;=\;
    \bigg|
      \al\, \int_{|b|}^{|a|} t^{\be-1}\, t^{\al-\be}\, \md t
    \bigg|
    \;\leq\;
    \bigg|\frac{\al}{\be}\bigg| \bigg(\max_{t \in [|b|, |a|]} t^{\al-\be}\bigg)
    \big| \tilde{a}^{\be} - \tilde{b}^{\be} \big|.
  \end{align*}
  Since $t \mapsto t^{\al-\be}$ for $t \geq 0$ is monotone decreasing if $\al -\be < 0$ and monotone increasing if $\al-\be > 0$, the maximum is attained at one of the boundary points.  In particular, $\max_{t \in [|b|, |a|]} t^{\al-\be} \leq |a|^{\al-\be} + |b|^{\al-\be}$.  On the other hand, for $a \geq 0, b\leq 0$ or $a \leq 0, b\geq 0$, it holds that $\big|\tilde{a}^{\al}-\tilde{b}^{\al}\big| = |a|^{\al} + |b|^{\al}$.  Hence,
  \begin{align*}
    \big|\tilde{a}^{\al} - \tilde{b}^{\al}\big|
    \;\leq\;
    \big(|a|^{\be} + |b|^{\be}\big)\big(|a|^{\al-\be} + |b|^{\al-\be}\big)
    \;=\;
    \big|\tilde{a}^{\be} - \tilde{b}^{\be}\big|
    \big(|a|^{\al-\be} + |b|^{\al-\be}\big).
  \end{align*}

  \textit{(ii)} Notice that for all $a,b \in \bbR$ and for any $\al > 0$,
  \begin{align*}
    \big(\tilde{a}^\al - \tilde{b}^\al\big)^2
    \;=\;
    \bigg(\al\, \int_b^a |t|^{\al-1} \, \md t\bigg)^{\!\!2}.
  \end{align*}
  Hence, for any $\al > 1/2$, an application of the Cauchy--Schwarz inequality yields
  \begin{align*}
    \bigg( \al \int_{b}^{a} |t|^{\al-1} \md t\bigg)^{\!\!2}
    \;\leq\;
    \al^2\, \int_{b}^{a} \md t\, \int_{b}^{a} |t|^{2\al-2} \md t
    \;=\;
    \frac{\al^2}{2\al-1}\, \big(a - b\big)\,
    \big(\tilde{a}^{2\al-1} - \tilde{b}^{2\al-1}\big).
  \end{align*}

  \textit{(iii)} Let us first assume that $a \geq 0, b \leq 0$ or $a \leq 0, b \geq 0$.  Then, for all $\al, \be \geq0$,
  \begin{align*}
    \big(|a|^{\al} +\, |b|^{\al}\big)\,
    \big|\tilde{a}^{\be} -\, \tilde{b}^{\be}\big|
    &\;=\;
    \big(|a|^{\al} + |b|^{\al}\big)\,
    \big(|a|^{\be} +\, |b|^{\be}\big)
    \\
    &\;=\;
    2\, \big(|a|^{\al+\be} +\, |b|^{\al+\be}\big) \,-\,
    \big(|a|^{\al} -\, |b|^{\al}\big)\big(|a|^{\be} -\, |b|^{\be}\big).    
  \end{align*}
  One sees easily that the last term is positive as long as $\al, \be \geq 0$.  It remains to consider the case when $a,b \geq 0$ or $a,b \leq 0$.  Since the assertion is trivial, if $a = 0$, we assume that $a \ne 0$ and set $z = b/a \geq 0$.  Then the left-hand side of \eqref{eq:A1:chain:lo} reads $\big(1 + z^{\al}\big)\, \big|1 - z^{\be}\big|$.  Provided that $\al \geq 0$, by distinguishing two cases, $z \in [0,1)$ or $z \geq 1$, we obtain
  \begin{align*}
    \big(1 + z^{\al}\big)\, \big|1 - z^{\be}\big|
    \;=\;
    2 \big|1 - z^{\al+\be}\big| \,-\, \big|1 - z^{\al}\big|\,
    \big(1 + z^{\be}\big)
    \;\leq\;
    2 \big|1 - z^{\al+\be}\big|.
  \end{align*}
  This completes the proof.

  \textit{(iv)} The assertion follows immediately from \eqref{eq:A1:chain:ub1} and \eqref{eq:A1:chain:lo}. 
\end{proof}
\subsubsection*{Acknowledgement} We thank Martin Barlow, Marek Biskup, Takashi Kumagai and Stefan Neukamm for useful discussions and valuable comments.  M.S. gratefully acknowledges financial support from the DFG Forschergruppe 718 ''Analysis and Stochastics in Complex Physical Systems''.

\bibliographystyle{plain}
\bibliography{literature}

\end{document}